\documentclass[12pt,a4paper]{article}

\usepackage{amsmath,amstext,amssymb,amscd}

\newcommand{\Xcomment}[1]{}

\newtheorem{theorem}{Theorem}[section]
\newtheorem{lemma}[theorem]{Lemma}

\newtheorem{prop}[theorem]{Proposition}

\makeatletter \@addtoreset{equation}{section} \makeatother

\newenvironment{proof}{\noindent{\bf Proof}\/}%
{\hfill$\qed$\medskip}

\def\qed{ \ \vrule width.1cm height.3cm depth0cm}

\newenvironment{numitem1}{\refstepcounter{equation}\begin{enumerate}%
\item[(\thesection.\arabic{equation})]}{\end{enumerate}}

\newcommand{\refeq}[1]{(\ref{eq:#1})}  % reference to equation

\def\rest#1{_{\,\vrule height 1.6ex width 0.05em depth 0pt\, #1}}

 \makeatletter
\renewcommand{\section}{\@startsection{section}{1}{0pt}%
{-3.5ex plus -1ex minus -.2ex}{2.3ex plus .2ex}%
{\normalfont\Large}}
 \makeatother

 \makeatletter
\renewcommand{\subsection}{\@startsection{subsection}{2}{0pt}%
{-3.0ex plus -1ex minus -.2ex}{1.5ex plus .2ex}%
{\normalfont\normalsize\bf}}
 \makeatother

\def\Rset{{\mathbb R}}
\def\Nset{{\mathbb N}}
\def\Zset{{\mathbb Z}}

\def\Ascr{{\cal A}}
\def\Bscr{{\cal B}}
\def\Cscr{{\cal C}}
\def\Dscr{{\cal D}}
\def\Fscr{{\cal F}}

\def\Iscr{{\cal I}}
\def\Kscr{{\cal K}}

\def\Mscr{{\cal M}}
\def\Pscr{{\cal P}}
\def\Qscr{{\cal Q}}

\def\frakL{{\mathfrak L}}

\def\frakS{{\mathfrak S}}

\def\tilde{\widetilde}
\def\hat{\widehat}
\def\bar{\overline}

\def\Bleft{B^{\rm left}}
\def\Bright{B^{\rm right}}

\def\deltain{\delta^{\rm in}}
\def\deltaout{\delta^{\rm out}}

\def\odivide{/}

\begin{document}

 \begin{center}
{\large\bf Planar flows and Pl\"ucker's type quadratic relations \\ over
semirings}
 \end{center}

 \begin{center}
{\sc Vladimir~I.~Danilov}\footnote[1] {Central Institute of Economics and
Mathematics of the RAS, 47, Nakhimovskii Prospect, 117418 Moscow, Russia;
emails: danilov@cemi.rssi.ru (V.I.~Danilov); koshevoy@cemi.rssi.ru
(G.A.~Koshevoy).},
{\sc Alexander~V.~Karzanov}\footnote[2]{Institute for System Analysis of the
RAS, 9, Prospect 60 Let Oktyabrya, 117312 Moscow, Russia; email:
sasha@cs.isa.ru.},
{\sc Gleb~A.~Koshevoy}$^1$
\end{center}

 \begin{center}
 17.08.2010
 \end{center}

%\maketitle

 \begin{quote}
 {\bf Abstract.} \small
It is well known, due to Lindstr\"om, that the minors of a (real or complex)
matrix can be expressed in terms of weights of flows in a planar directed
graph. Another classical fact is that there are plenty of homogeneous quadratic
relations involving flag minors, or Pl\"ucker coordinates of the corresponding
flag manifold. Generalizing and unifying these facts and their tropical
counterparts, we consider a wide class of functions on $2^{[n]}$ that are
generated by flows in a planar graph and take values in an arbitrary
commutative semiring, where $[n]=\{1,2,\ldots,n\}$. We show that the
``universal'' homogeneous quadratic relations fulfilled by such functions can
be described in terms of certain matchings, and as a consequence, give
combinatorial necessary and sufficient conditions on the collections of subsets
of $[n]$ determining these relations.


  \end{quote}

\parskip=3pt

%----------------------- Sec. 1

\section{\Large Introduction}  \label{sec:intr}

For a positive integer $n$, let $[n]$ denote the set of integers
$1,2,\ldots,n$.

In this paper we consider functions on the set $2^{[n]}$ of subsets of $[n]$
(or the $n$-dimensional Boolean cube) that take values in a commutative
semiring and are generated by planar flows. Functions of this sort satisfy
plenty of quadratic relations of Pl\"ucker's type, and our goal is to describe
a combinatorial method that enables us to reveal and easily prove such
relations.

We start with recalling some basic facts concerning Pl\"ucker algebra and
Pl\"ucker coordinates. Consider the $n\times n$ matrix $\bf x$ of
indeterminates $x_{ij}$ and its associated polynomial ring $\Zset[\bf x]$. Also
consider the polynomial ring $\Zset[\Delta]$ associated to the set of $2^n$
variables $\Delta_S$ indexed by the subsets $S\subseteq [n]$. They are linked
by the natural ring homomorphism $\psi:\Zset[\Delta]\to\Zset[\bf x]$ that
brings each variable $\Delta_S$ to the flag minor polynomial for $S$, i.e. to
the determinant of the submatrix $\bf{x}_S$ formed by the column set $S$ and
the row set $\{1,\ldots,|S|\}$ of $\bf x$. An important fact is that the ideal
${\rm ker}(\psi)$ of $\Zset[\Delta]$ is generated by some homogeneous quadrics,
each being an integer combination of products $\Delta_S\Delta_{S'}$ with the
same parameter $(|S|,|S'|)$. They correspond to quadratic relations on the
Pl\"ucker coordinates of a (real say) invertible $n\times n$ matrix (viz. on
the Pl\"ucker coordinates of a point of the flag manifold over $\Rset^n$
embedded in the appropriate projective space); for a survey see,
e.g.,~\cite[Ch.~14]{MS}.

There are plenty of quadratic Pl\"ucker relations on flag minors of a matrix
whose entries are assumed to belong to an arbitrary commutative ring
$\mathfrak{R}$ (of which the case $\mathfrak{R}=\Rset$ or $\mathbb{C}$ is most
popular). Let $f(S)$ denote the flag minor with a column set $S$ in this
matrix.

A well-known (and the simplest) special case of Pl\"ucker relations involves
triples: for any three elements $i<j<k$ in $[n]$ and any subset
$X\subseteq[n]-\{i,j,k\}$, the flag minor function $f:2^{[n]}\to \mathfrak{R}$
of an $n\times n$ matrix satisfies
   \begin{equation} \label{eq:AP3}
    f(Xik)f(Xj)-f(Xij)f(Xk)-f(Xjk)f(Xi)=0,
    \end{equation}
where for brevity we write $Xi'\ldots j'$ for $X\cup\{i',\ldots,j'\}$. We refer
to~\refeq{AP3} as the \emph{AP3-relation} (abbreviating ``algebraic Pl\"ucker
relation with triples").  Another well-known special case (in particular,
encountered in a characterization of Grassmannians) involves quadruples
$i<j<k<\ell$ and is viewed as
   \begin{equation} \label{eq:AP4}
    f(Xik)f(Xj\ell)-f(Xij)f(Xk\ell)-f(Xi\ell)f(Xjk)=0.
    \end{equation}

A general (algebraic quadratic) Pl\"ucker relation on flag minors of a matrix
can be written in the form
  \begin{equation} \label{eq:a_pluck}
  \sum_{A\in\Ascr} f(X\cup \gamma_Y(A))f(X\cup\gamma_Y(\bar A))-
     \sum_{A'\in\Ascr'} f(X\cup \gamma_Y(A'))f(X\cup\gamma_Y(\bar A'))=0.
  \end{equation}
Here: (a) $\Ascr$ and $\Ascr'$ are certain collections of $p$-element subsets
in $[p+q]$ for some integers $p,q>0$ with $p+q\le n$; (b)~$Y$ is a
$(p+q)$-element subset in $[n]$, and $\gamma=\gamma_Y$ is the order preserving
bijective map $[p+q]\to Y$ (i.e. $\gamma(i)<\gamma(j)$ for $i<j$); (c)~$X$ is
an arbitrary subset of $[n]-Y$; and (d) $\bar A$ stands for the complement
$[p+q]-A$ of $A\subseteq [p+q]$. Emphasize that~\refeq{a_pluck} should be valid
for the flag minor function $f$ of any $n\times n$ matrix (over any
$\mathfrak{R}$) and depends only on $p,q,\Ascr,\Ascr'$ but not $X,Y$. Note also
that each of $\Ascr,\Ascr'$ is admitted to be a collection in which multiple
sets $A\subseteq[p+q]$ are allowed (sometimes called a \emph{multicollection});
in spite of this, to simplify notation we will write $\Ascr,\Ascr' \subseteq
\binom{[p+q]}{p}$.

In particular,~\refeq{a_pluck} turns into~\refeq{AP3} when $p=2$, $q=1$,
$\Ascr=\{13\}$, $\Ascr'=\{12,23\}$ and $Y=\{i,j,k\}$, and turns
into~\refeq{AP4} when $p=q=2$, $\Ascr=\{13\}$, $\Ascr'=\{12,14\}$ and
$Y=\{i,j,k,\ell\}$.

An important property shown by Lindstr\"om~\cite{Li} is that the minors of a
matrix can be expressed by use of flows in a planar graph. A flow model will be
the focus of our further description, and we now specify the notion of planar
flows that we deal with. (See also~\cite{RS} for further applications of the
flow model.)

By a \emph{planar network} we mean a finite directed planar graph $G=(V,E)$
(properly embedded in the plane) in which two $n$-element subsets
$S=\{s_1,\ldots,s_n\}$ and $T=\{t_1,\ldots,t_n\}$ of vertices are
distinguished, called the sets of \emph{sources} and \emph{sinks} in $G$,
respectively. We throughout assume that (a) $G$ is (weakly) connected, that (b)
the sources and sinks belong to the boundary (of the outer face) of $G$ and
occur in it in the cyclic order $s_n,\ldots,s_1,t_1,\ldots,t_n$ (with possibly
$s_1=t_1$ or $s_n=t_n$), and that (c) $G$ is \emph{acyclic}, i.e. contains no
directed cycle. We will attribute the term ``network'' to the graph $G$ itself
when the sets of sources and sinks in it are clear from the context. An
important particular case is the \emph{half-grid} $\Gamma_n$ whose vertices are
the integer points $(i,j)\in\Rset^2$ with $1\le j\le i\le n$, the edges are all
possible ordered pairs of the form $((i,j),(i-1,j))$ or $((i,j),(i,j+1))$, the
sources are $s_i=(i,1)$ and the sinks are $t_i=(i,i)$, $i=1,\ldots,n$. The
half-grid $\Gamma_{5}$ is illustrated in Fig.~\ref{fig:Gamma}.

\begin{figure}[htb]
 \begin{center}
 \unitlength=.8mm %\special{em:linewidth 0.4pt}
\linethickness{0.4pt}
\begin{picture}(81.00,50.00)(20,5)
\put(40.00,5.00){\circle{2}} \put(50.00,5.00){\circle{2}}
\put(60.00,5.00){\circle{2}} \put(70.00,5.00){\circle{2}}
\put(80.00,5.00){\circle{2}} \put(50.00,15.00){\circle{2}}
\put(60.00,15.00){\circle{2}} \put(70.00,15.00){\circle{2}}
\put(80.00,15.00){\circle{2}} \put(60.00,25.00){\circle{2}}
\put(70.00,25.00){\circle{2}} \put(80.00,25.00){\circle{2}}
\put(70.00,35.00){\circle{2}} \put(80.00,35.00){\circle{2}}
\put(80.00,45.00){\circle{2}} \put(79.00,5.00){\vector(-1,0){8.00}}
\put(69.00,5.00){\vector(-1,0){8.00}} \put(59.00,5.00){\vector(-1,0){8.00}}
\put(49.00,5.00){\vector(-1,0){8.00}} \put(79.00,15.00){\vector(-1,0){8.00}}
\put(69.00,15.00){\vector(-1,0){8.00}} \put(59.00,15.00){\vector(-1,0){8.00}}
\put(79.00,25.00){\vector(-1,0){8.00}} \put(69.00,25.00){\vector(-1,0){8.00}}
\put(79.00,35.00){\vector(-1,0){8.00}} \put(80.00,6.00){\vector(0,1){8.00}}
\put(80.00,16.00){\vector(0,1){8.00}} \put(80.00,26.00){\vector(0,1){8.00}}
\put(80.00,36.00){\vector(0,1){8.00}} \put(70.00,6.00){\vector(0,1){8.00}}
\put(70.00,16.00){\vector(0,1){8.00}} \put(70.00,26.00){\vector(0,1){8.00}}
\put(60.00,6.00){\vector(0,1){8.00}} \put(60.00,16.00){\vector(0,1){8.00}}
\put(50.00,6.00){\vector(0,1){8.00}} \put(40.00,1.00){\makebox(0,0)[cc]{$s_1$}}
\put(50.00,1.00){\makebox(0,0)[cc]{$s_2$}}
\put(60.00,1.00){\makebox(0,0)[cc]{$s_3$}}
\put(70.00,1.00){\makebox(0,0)[cc]{$s_4$}}
\put(80.00,1.00){\makebox(0,0)[cc]{$s_5$}}
\put(36.00,7.00){\makebox(0,0)[cc]{$t_1$}}
\put(47.00,17.00){\makebox(0,0)[cc]{$t_2$}}
\put(57.00,27.00){\makebox(0,0)[cc]{$t_3$}}
\put(67.00,37.00){\makebox(0,0)[cc]{$t_4$}}
\put(76.50,46.00){\makebox(0,0)[cc]{$t_5$}}
\end{picture}
  \end{center}
\caption{The half-grid $\Gamma_5$} \label{fig:Gamma}
  \end{figure}
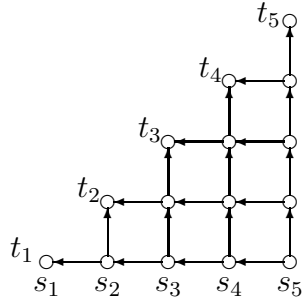

For $I\subseteq[n]$, define $S_I:=\{s_i\colon i\in I\}$ and $T_I:=\{t_i\colon
i\in I\}$. Speaking of an $(I,J)$-\emph{flow}, where $I,J\subseteq[n]$ and
$|I|=|J|=:k$, we mean a collection $\phi$ of $k$ pairwise (vertex) disjoint
directed paths in $G$ going from the source set $S_I$ to the sink set $T_J$.
When $\phi$ enters the first $k$ sinks (i.e. $J=[k]$), we refer to $\phi$ as a
\emph{flag flow} for $I$, or an $I$-\emph{flow}. The set of $I$-flows (resp.
$(I,J)$-flows) in $G$ is denoted by $\Phi_I=\Phi_I^G$ (resp.
$\Phi_{I,J}=\Phi_{I,J}^G$).

Let $w:V\to\mathfrak{R}$ be a weighting on the vertices of $G$, where, as
before, $\mathfrak{R}$ is a commutative ring. We associate to $w$ the function
$f=f_w$ on $2^{[n]}$ defined by
  \begin{equation} \label{eq:alg_f}
  f(I):=\sum\nolimits_{\phi\in\Phi_I}\prod\nolimits_{v\in V_\phi} w(v),
  \qquad I\subseteq [n],
  \end{equation}
where $V_\phi$ is the set of vertices occurring in a flow $\phi$. (It is
possible that $G$ has no flag flow for some $I$, in which case $f(I)$ becomes
0.) We refer to $f$ obtained in this way as an \emph{algebraic flow-generated
function}, or an \emph{AFG-function} for short. By Lindstr\"om
theorem~\cite{Li}, if $M$ is the $n\times n$ matrix whose entries $m_{ji}$ are
defined as $\sum_{\phi\in\Phi_{\{i\},\{j\}}}\prod_{v\in V_\phi} w(v)$, then for
any $I,J\subseteq [n]$ with $|I|=|J|$, the minor of $M$ with the column set $I$
and the row set $J$ is equal to $f(I,J)$, where the latter is defined as in
expression~\refeq{alg_f} with $\Phi_I$ replaced by $\Phi_{I,J}$. A converse
property takes place as well (at least for $\mathfrak{R}=\Rset$ or
$\mathbb{C}$): the minors of any $n\times n$ matrix can be expressed as above
via flows for some planar network and weighting.

Another important application of the flow model concerns tropical analogues of
the above quadratic relations. In this case the flow-generated function $f=f_w$
on $2^{[n]}$ determined by a weighting $w$ on $V$ is defined as
   \begin{equation} \label{eq:trop_f}
  f(I):=\max_{\phi\in\Phi_I}\left(\sum\nolimits_{v\in V_\phi} w(v)\right),
  \qquad I\subseteq [n].
  \end{equation}
Here $w$ is assumed to take values in a totally ordered abelian group
$\mathfrak{L}$ (usually one deals with $\mathfrak{L}=\Rset$ or $\Zset$). The
formula for $f$ in~\refeq{trop_f} is nothing else than the tropicalization of
that in~\refeq{alg_f}, and $f$ is said to be a \emph{tropical flow-generated
function}, or a \emph{TFG-function}. Some appealing properties of such
functions and related objects are demonstrated in~\cite{DKK1}. (See
also~\cite{DKK2} for additional results. Note that~\cite{DKK1,DKK2} deal with
real-valued tropical functions but everywhere $\Rset$ can be replaced by
$\mathfrak{L}$.) In particular, one shows that a TFG-function $f$ satisfies the
tropical analog of~\refeq{AP3}, or the \emph{TP3-relation}:
   \begin{equation} \label{eq:TP3}
    f(Xik)+f(Xj)=\max\{f(Xij)+f(Xk),f(Xjk)+f(Xi)\},
    \end{equation}
where, as before, $i<j<k$ and $X\subseteq[n]-\{i,j,k\}$. It turns out that a
converse property holds as well: any function $f:2^{[n]}\to\mathfrak{L}$
obeying the TP3-relation (for all $i,j,k,X$) is a TFG-function (determined by
some $G$ and $w$). In fact, to generate the set ${\bf T}_n(\mathfrak{L})$ of
all TFG-functions it suffices to consider only one planar network, namely, the
above-mentioned half-grid $\Gamma_n=(V,E)$, which emphasizes an important role
of the latter. More precisely, the correspondence $w\mapsto f_w$, where
$w:V\to\mathfrak{L}$, gives a bijection between $\mathfrak{L}^V$ and ${\bf
T}_n(\mathfrak{L})$. Two more results shown in~\cite{DKK1} by handling flows in
$\Gamma_n$ are:

(i) ${\bf T}_n(\mathfrak{L})$ has as a sort of \emph{basis} the set $\Iscr_n$
of all intervals $[p..q]:=\{p,p+1,\ldots,q\}$ in $[n]$ (including the ``empty
interval'' $\emptyset$). This means that the restriction map $f\mapsto
{f}\rest{\Iscr_n}$ gives a bijection between ${\bf T}_n(\mathfrak{L})$ and
$\mathfrak{L}^{\Iscr_n}$ (i.e. any TFG-function is determined by its values on
the intervals, and those values can be chosen arbitrarily in $\mathfrak{L}$);

(ii) for any subset $X\subseteq[n]$, the value of a TFG-function $f$ on $X$ can
be expressed by a tropical Laurent polynomial in variables $f(I)$,
$I\in\Iscr_n$.

(Note that in general ${\bf T}_n(\mathfrak{L})$ admits many bases as in~(i); an
especial role of the basis $\Iscr_n$ is discussed in~\cite{DKK1} where this
basis is called \emph{standard}. Note also that (ii) gives a tropical analogue
of the Laurentness phenomenon for algebraic flow-generated functions $f$, i.e.
the values of $f$ are Laurent polynomials in the values on intervals, in the
assumption that the latter ones are positive; see~\cite{FZ}.)
\medskip

In this paper we combine both algebraic and tropical cases by considering
functions taking values in an arbitrary \emph{commutative semiring}
$\mathfrak{S}$, a set equipped with two associative and commutative binary
operations $\oplus$ (addition) and $\odot$ (multiplication) satisfying the
distributive law $a\odot(b\oplus c)=(a\odot b)\oplus(a\odot c)$. Sometimes we
assume, in addition, that $\mathfrak{S}$ contains neutral elements $\underline
0$ (for addition) and/or $\underline 1$ (for multiplication). Two special cases
are of especial interest for us. When $\underline 0\in\mathfrak{S}$ and each
element has an additive inverse, $\mathfrak{S}$ becomes a commutative ring as
above. Another case is a commutative semiring with division, i.e. $\underline
1\in\mathfrak{S}$ and each element has a multiplicative inverse. Examples of
the latter are: the set $\Rset_{>0}$ of positive reals (with $\oplus=+$ and
$\odot=\cdot$), and the above-mentioned tropicalization of a totally ordered
abelian group $\mathfrak{L}$, denoted as $\mathfrak{L}^{\rm trop}$ (with
$\oplus=\max$ and $\odot=+$).

Extending~\refeq{alg_f} and~\refeq{trop_f}, we define the flow-generated
function $f=f_w$ determined by a weighting $w:V\to\mathfrak{S}$ as
  \begin{equation} \label{eq:S_f}
  f(I):=\bigoplus\nolimits_{\phi\in\Phi_I} w(\phi),
  \qquad I\subseteq [n],
  \end{equation}
where $w(\phi)$ stands for the weight $\odot(w(v)\colon v\in V_\phi)$ of a flow
$\phi$. We call $f$ an \emph{SFG-function} (abbreviating ``flow-generated
function over a semiring''), and denote the set of these functions by ${\bf
FG}={\bf FG}_n(\mathfrak{S})$. A direct analogue of identity~\refeq{a_pluck}
for $\mathfrak{S}$ is viewed as
  \begin{multline} \label{eq:S_pluck}
  \bigoplus\nolimits_{A\in\Ascr} \left(f(X\cup \gamma_Y(A))\odot
      f(X\cup\gamma_Y(\bar A))\right)     \\
 =\bigoplus\nolimits_{A'\in\Ascr'} \left(f(X\cup \gamma_Y(A'))\odot
   f(X\cup\gamma_Y(\bar A'))\right),
  \end{multline}
and when this holds true for fixed (multi)collections $\Ascr,\Ascr'$ and for
any corresponding $\mathfrak{S},G,w,X,Y$, we say that~\refeq{S_pluck} is a
\emph{stable quadratic relation}, or an \emph{sq-relation}. \medskip

 \noindent\textbf{Remark 1.}
If $G$ and $I$ are such that $\Phi^G_I=\emptyset$, then~\refeq{S_f} is not
applicable in general. In this case $f(I)$ may be regarded as \emph{undefined},
and whenever expression~\refeq{S_pluck} contains a summand $f(I)\odot f(J)$
with at least one of $f(I),f(J)$ being undefined (for the given $G$), we may
think that this summand simply vanishes in the expression. (An alternative way
is to put $f(I)$ to be an ``extra neutral'' element $\ast$ added to
$\mathfrak{S}$, setting $\ast\oplus a=a$ and $\ast\odot a=\ast$ for all $a\in
\mathfrak{S}$.) In particular, $f(\emptyset)$ may be regarded as undefined, and
we will usually ignore the value of an SFG-function on the element $\emptyset$.
\medskip

The goal of this paper is to describe a relatively simple combinatorial method
of constructing pairs $\Ascr,\Ascr'$ determining sq-relations, and we give
necessary and sufficient conditions on such pairs. In fact, our method is
inspired by flow rearranging techniques elaborated in~\cite{DKK1} for proving
the TP3-relation for TFG-functions. The method reduces the task to a
combinatorial problem of smaller size (and provides a polynomial-time algorithm
to recognize whether or not a pair $\Ascr,\Ascr'$ gives an sq-relation). This
is exposed in Theorem~\ref{tm:main} which bridges validity of~\refeq{S_pluck}
for $p,q,\Ascr,\Ascr'$ and the property that two collections of certain
\emph{matchings} associated to $\Ascr,\Ascr'$ are \emph{balanced}; the meaning
of the latter notion will be explained later. It should be noted that our
method of handling flows resembles, to some extent, a technique in~\cite{Ku}
where quadratic relations on the amounts of perfect matchings in certain
subgraphs of a planar graph are established.

The paper is organized as follows. Section~\ref{sec:flow} describes properties
of certain pairs of flows (\emph{double flows}) which lie in the background of
our method. Section~\ref{sec:balan} states the main result
(Theorem~\ref{tm:main}) and proves the sufficiency part in it, claiming that
all balanced collections $\Ascr,\Ascr'$ generate sq-relations.
Section~\ref{sec:relat} is devoted to illustrations of the method, which
demonstrate a number of particular and wider classes of stable
identities~\refeq{S_pluck}. Section~\ref{sec:(i)-(ii)} proves the necessity
part in the main theorem; moreover, we show that if collections $\Ascr,\Ascr'$
are not balanced, then the corresponding quadratic relation does not hold
already for some AFG-function with $\mathfrak{R}=\Rset$. This implies that
\emph{for $\Ascr,\Ascr'$ fixed, validity of~\refeq{S_pluck} for all $\frakS$ is
equivalent to validity of~\refeq{a_pluck} for $\Rset$.} (This responds the
so-called \emph{transfer principle} for semirings; see, e.g.,
\cite[Sec.~3]{AGG}.) The final Section~\ref{sec:laurent} contains a short
discussion on the standard basis and the Laurent phenomenon for SFG-functions
over a commutative semiring with division.

%----------------------- Sec.2

\section{\Large Flows and double flows}  \label{sec:flow}

Let $G=(V,E)$ be a planar network with sources $s_1,\ldots,s_n$ and sinks
$t_1,\ldots,t_n$ arranged as above, and let $p,q\in\Nset$ and $p+q\le n$. As
before, we assume that $G$ is (weakly) connected and acyclic. In this section
we describe ideas and tools behind the method of constructing
(multi)collections $\Ascr,\Ascr'\subseteq\binom{[p+q]}{p}$ that ensure validity
of~\refeq{S_pluck} for all flow-generated functions $f=f_w$ on $2^{[n]}$
determined by weightings $w:V\to \mathfrak{S}$, where $\mathfrak{S}$ is an
arbitrary commutative semiring.

First of all we specify some terminology and notation. By a \emph{path} in a
digraph (directed graph) we mean a sequence $P=(v_0,e_1,v_1,\ldots,e_k,v_k)$
where each $e_i$ is an edge connecting vertices $v_{i-1},v_i$. An edge $e_i$ is
called \emph{forward} if it is directed from $v_{i-1}$ to $v_i$, denoted as
$e_i=(v_{i-1},v_i)$, and \emph{ backward} otherwise (when $e_i=(v_i,v_{i-1})$).
The path $P$ is called {\em directed} if it has only forward edges, and {\em
simple} if all vertices $v_i$ are distinct. When $k>0$, $v_0=v_k$ and all
$v_1,\ldots,v_k$ are distinct, ~$P$ is called a \emph{simple cycle}, or a
\emph{circuit}; it is often considered up to cyclically shifting and reversing.
The sets of vertices and edges of $P$ are denoted by $V_P$ and $E_P$,
respectively.

Recall that by an $I$-flow in $G$, where $I\subseteq[n]$, we mean a collection
$\phi$ of $|I|$ pairwise disjoint directed paths going from the source set
$S_I$ to the sink set $\{t_1,\ldots,t_{|I|}\}$. Since $G$ is acyclic, all these
paths are simple, and the order of sources and sinks in the boundary of $G$
implies that the path in $\phi$ entering a sink $t_i$ begins at $i$-th source
in $S_I$ (in the natural ordering there). A useful equivalent definition of an
$I$-flow $\phi$ is that $\deltaout_\phi(s_i)=1$ and $\deltain_\phi(s_i)=0$ if
$i\in I$; $\deltaout_\phi(t_j)=0$ and $\deltain_\phi(t_j)=1$ if $1\le j\le
|I|$; and $\deltaout_\phi(v)=\deltain_\phi(v)\in\{0,1\}$ for the other vertices
$v$ in $G$. Here $\deltaout_\phi(v)$ (resp. $\deltain_\phi(v)$) denotes the
number of edges in $\phi$ leaving (resp. entering) a vertex $v$. Also we denote
$\deltaout_\phi(v)+\deltain_\phi(v)$ by $\delta_\phi(v)$.

Our approach is based on examining certain pairs of flag flows in $G$ and
rearranging them to form some other pairs. To simplify technical details, it is
convenient to consider an equivalent flow model, obtained by slightly modifying
the network $G$, as follows. Let us split each vertex $v\in V$ into two
vertices $v',v''$ (disposing them in a small neighborhood of $v$ in the plane)
and connect them by edge $e_v=(v',v'')$, called a \emph{split-edge}. Each edge
$(u,v)$ of $G$ is replaced by an edge going from $u''$ to $v'$; we call it an
\emph{ordinary} edge. Also for each $s_i\in S$, we add new source $\hat s_i$
and edge $(\hat s_i,s')$, and for each $t_j\in T$, add new sink $\hat t_j$ and
edge $(t''_j,\hat t_j)$; we refer to such edges as \emph{extra} ones. The
picture illustrates the transformation for $G=\Gamma_3$.

 \begin{center}
 \unitlength=1.0mm
 %\linethickness{0.4pt}
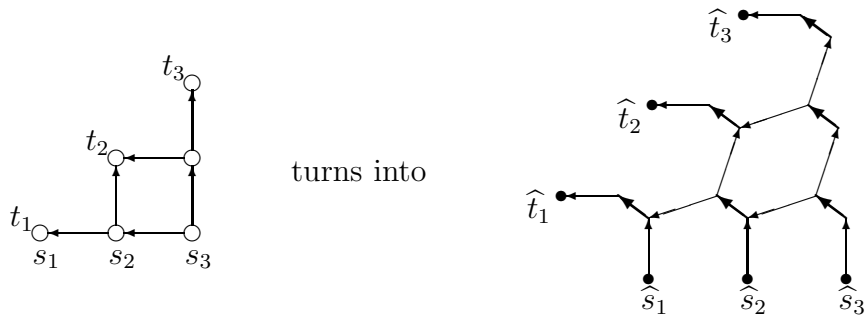
  \begin{picture}(120.00,40)
                         %              left
   \put(0,5){\begin{picture}(35,28)
 \put(10,5){\circle{2}}
 \put(20,5){\circle{2}}
 \put(30,5){\circle{2}}
 \put(20,15){\circle{2}}
 \put(30,15){\circle{2}}
 \put(30,25){\circle{2}}
 \put(19,5){\vector(-1,0){8}}
 \put(29,5){\vector(-1,0){8}}
 \put(29,15){\vector(-1,0){8}}
 \put(20,6){\vector(0,1){8}}
 \put(30,6){\vector(0,1){8}}
 \put(30,16){\vector(0,1){8}}
 \put(9,1){$s_1$}
 \put(19,1){$s_2$}
 \put(29,1){$s_3$}
 \put(6,6){$t_1$}
 \put(16,16){$t_2$}
 \put(26,26){$t_3$}
  \end{picture}}
 \put(43,17){turns into}
                         %              right
   \put(70,4){\begin{picture}(50,32)
 \put(20,0){\circle*{1.5}}
 \put(33,0){\circle*{1.5}}
 \put(46,0){\circle*{1.5}}
 \put(8.5,11){\circle*{1.5}}
 \put(20.5,23){\circle*{1.5}}
 \put(32.5,35){\circle*{1.5}}
 \put(20,0){\vector(0,1){8}}
 \put(33,0){\vector(0,1){8}}
 \put(46,0){\vector(0,1){8}}
 \put(16,11){\vector(-1,0){7}}
 \put(28,23){\vector(-1,0){7}}
 \put(40,35){\vector(-1,0){7}}
 \put(29,11){\vector(-3,-1){9}}
 \put(42,11){\vector(-3,-1){9}}
 \put(41,23){\vector(-3,-1){9}}
 \put(29,11){\vector(1,3){3}}
 \put(42,11){\vector(1,3){3}}
 \put(41,23){\vector(1,3){3}}
{\thicklines
 \put(20,8){\vector(-4,3){4}}
 \put(33,8){\vector(-4,3){4}}
 \put(46,8){\vector(-4,3){4}}
 \put(32,20){\vector(-4,3){4}}
 \put(45,20){\vector(-4,3){4}}
 \put(44,32){\vector(-4,3){4}}
 }
 \put(19,-4){$\hat s_1$}
 \put(32,-4){$\hat s_2$}
 \put(45,-4){$\hat s_3$}
 \put(4,8){$\hat t_1$}
 \put(16,20){$\hat t_2$}
 \put(28,32){$\hat t_3$}
  \end{picture}}
 \end{picture}
  \end{center}

Note that the new (modified) network is again acyclic, but it need not be
planar in general. Nevertheless, we keep the same notation $G=(V,E)$ for it,
and take $\hat S:=\{\hat s_1,\ldots,\hat s_n\}$ and $\hat T:=\{\hat
t_1,\ldots,\hat t_n\}$ as the sets of sources and sinks in it, respectively.
Sources and sinks are also called \emph{terminal} vertices. Clearly for any
$i,j\in[n]$, there is a natural 1--1 correspondence between the directed paths
from $s_i$ to $t_j$ in the initial network and the ones from $\hat s_i$ to
$\hat t_j$ in the modified network. This is extended to a natural 1--1
correspondence between flag flows, and for $I\subseteq[n]$, we keep notation
$\Phi_I$ for the set of flows going from $\hat S_I:=\{\hat s_i \colon i\in I\}$
to $\hat T_{|I|}:=\{\hat t_1,\ldots,\hat t_{|I|}\}$. A weighting $w$ on the
vertices $v$ of the initial $G$ is transferred to the split-edges of the
modified $G$, namely, $w(e_v):=w(v)$. Then corresponding flows in both networks
have equal weights (which are the products by $\odot$ of the weights of
vertices or split-edges in the flows). This implies that the functions on
$2^{[n]}$ generated by corresponding flows coincide.

We will take advantages from the following obvious property of the modified
$G$:
  \begin{numitem1}
each non-terminal vertex $u$ is incident with exactly one split-edge $e$, and
if $e$ enters (leaves) $u$, then  $\deltain_G(u)=1$ (resp. $\deltaout_G(u)=1$);
each terminal vertex has exactly one incident edge.
  \label{eq:1edge}
  \end{numitem1}

Let $X,Y$ be disjoint subsets of $[n]$ and $|Y|=p+q$. We assume that $p\ge q$
and, as before, denote by $\gamma=\gamma_Y$ the order preserving bijective map
of $[p+q]$ to $Y$. We write $C\triangle D$ for the symmetric difference
$(C-D)\cup(D-C)$ of subsets $C,D$ of a set.

Let us fix a subset $A\in\binom{[p+q]}{p}$. Define $I(A):=X\cup\gamma(A)$ and
$J(A):=X\cup\gamma(\bar A)$, where $\bar A:=[p+q]-A$. Consider an $I(A)$-flow
$\phi$ and a $J(A)$-flow $\phi'$ in $G$. Our method will rely on the following
three lemmas.

  \begin{lemma} \label{lm:sumFFp}
~$E_\phi\triangle E_{\phi'}$ is partitioned into the edge sets of pairwise
disjoint circuits $C_1,\ldots,C_d$ (for some $d$) and simple paths $P_1,\ldots,
P_p$, where each $P_i$ connects a source in $\hat S_{\gamma(A)}$ with either a
source in $\hat S_{\gamma(\bar A)}$ or a sink in the set $\tilde T:=\hat
T_{|X|+p}-\hat T_{|X|+q}$. In each of these circuits and paths, the edges of
$\phi$ and the edges of $\phi'$ have opposed directions (say, the former edges
are forward and the latter ones are backward).
  \end{lemma}
  \begin{proof}
~Observe that a vertex $v$ of $G$ satisfies: (i) $\delta_{\phi}(v)=1$ and
$\delta_{\phi'}(v)=0$ if $v\in\hat S_{\gamma(A)}\cup \tilde T$; (ii)
$\delta_{\phi}(v)=0$ and $\delta_{\phi'}(v)=1$ if $v\in\hat S_{\gamma(\bar
A)}$; (iii) $\delta_{\phi}(v)=\delta_{\phi'}(v)=1$ if $v\in \hat S_X\cup \hat
T_{|X|+q}$; and (iv) $\delta_{\phi}(v),\delta_{\phi'}(v)\in\{0,2\}$ otherwise.
This together with property~\refeq{1edge} implies that any vertex $v$ is
incident with 0,1 or 2 edges in $E_\phi\triangle E_{\phi'}$, and the number is
equal to 1 if and only if $v\in \hat S_{\gamma(A)}\cup\hat S_{\gamma(\bar A)}
\cup \tilde T$. Hence the weakly connected components of the subgraph of $G$
induced by $E_\phi\triangle E_{\phi'}$ are circuits, $C_1,\ldots,C_d$ say, and
simple paths $P_1,\ldots,P_p$, each of the latter connecting two vertices in
$\hat S_{\gamma(A)}\cup\hat S_{\gamma(\bar A)} \cup \tilde T$.

Consider consecutive edges $e,e'$ in a circuit $C_i$ or a path $P_j$. If both
$e,e'$ belong to the same flow among $\phi,\phi'$, then, obviously, they have
the same direction in this circuit/path. Suppose $e,e'$ belong to different
flows. In view of~\refeq{1edge}, the common vertex $v$ of $e,e'$ is
non-terminal and incident with a split-edge $e''$. Clearly $e''$ belongs to
both $\phi,\phi'$, and therefore $e''\ne e,e'$. This implies that either both
$e,e'$ enter $v$ or both leave $v$, so they are directed differently along the
circuit/path containing them. This yields the second assertion in the lemma.

Finally, suppose some path $P_j$ has both ends in $\hat S_{\gamma(A)}$. Then
the first and last edges $e,e'$ of $P_j$ are extra edges of $G$ contained in
$\phi$. But both $e,e'$ leave $\hat S_{\gamma(A)}$, so they are directed
differently along $P_j$, contrary to proved above. Thus, each path $P_j$ has
exactly one end in $\hat S_{\gamma(A)}$ (in view of $|\hat S_{\gamma(A)}|=|\hat
S_{\gamma(\bar A)}|+|\tilde T|$), completing the proof.
  \end{proof}

Figure~\ref{fig:double} illustrates an example of $G,\phi,\phi',\xi$, and
$E_\phi\triangle E_{\phi'}$.

\begin{figure}[htb]

 \begin{center}
 \unitlength=1.0mm
  \begin{picture}(145,53)
   \put(0,10){\begin{picture}(20,43)  %      -  (a)
 \put(0,0){\circle{2}}
 \put(8,0){\circle{2}}
 \put(16,0){\circle{2}}
 \put(0,38){\circle{2}}
 \put(12,38){\circle{2}}
 \put(18,38){\circle{2}}
 \put(0,1){\vector(0,1){28.5}}
 \put(8,0){\vector(2,3){4}}
 \put(16,0){\vector(-2,3){4}}
 \put(12,6){\vector(0,1){6}}
 \put(12,12){\vector(1,1){6}}
 \put(12,12){\vector(-1,1){6}}
 \put(6,18){\vector(1,1){6}}
 \put(18,18){\vector(-1,1){6}}
 \put(12,24){\vector(0,1){6}}
 \put(12,30){\vector(-1,0){12}}
 \put(12,30){\vector(0,1){7}}
 \put(0,30){\vector(0,1){7}}
 \put(0,-5){$\hat s_1$}
 \put(8,-5){$\hat s_2$}
 \put(16,-5){$\hat s_3$}
 \put(0,40){$\hat t_1$}
 \put(9,40){$\hat t_2$}
 \put(18,40){$\hat t_3$}
  \end{picture}}
 \put(7,-2){(a)}
   \put(32,10){\begin{picture}(20,43)  %      -  (b)
 \put(0,0){\circle{2}}
 \put(8,0){\circle{2}}
 \put(16,0){\circle{2}}
 \put(0,38){\circle{2}}
 \put(12,38){\circle{2}}
 \put(18,38){\circle{2}}
 \put(0,1){\vector(0,1){28.5}}
 \put(16,0){\vector(-2,3){4}}
 \put(12,6){\vector(0,1){6}}
 \put(12,12){\vector(1,1){6}}
 \put(18,18){\vector(-1,1){6}}
 \put(12,24){\vector(0,1){6}}
 \put(12,30){\vector(0,1){7}}
 \put(0,30){\vector(0,1){7}}
  \end{picture}}
 \put(39,-2){(b)}
   \put(64,10){\begin{picture}(20,43)  %      -  (c)
 \put(0,0){\circle{2}}
 \put(8,0){\circle{2}}
 \put(16,0){\circle{2}}
 \put(0,38){\circle{2}}
 \put(12,38){\circle{2}}
 \put(18,38){\circle{2}}
 \put(8,0){\vector(2,3){4}}
 \put(12,6){\vector(0,1){6}}
 \put(12,12){\vector(-1,1){6}}
 \put(6,18){\vector(1,1){6}}
 \put(12,24){\vector(0,1){6}}
 \put(12,30){\vector(-1,0){12}}
 \put(0,30){\vector(0,1){7}}
  \end{picture}}
 \put(69,-2){(c)}
   \put(96,10){\begin{picture}(20,43)  %      -  (d)
 \put(0,0){\circle{2}}
 \put(8,0){\circle{2}}
 \put(16,0){\circle{2}}
 \put(0,38){\circle{2}}
 \put(12,38){\circle{2}}
 \put(18,38){\circle{2}}
 \put(0,1){\vector(0,1){28.5}}
 \put(8,0){\vector(2,3){4}}
 \put(16,0){\vector(-2,3){4}}
 \put(11.5,6){\vector(0,1){6}}
 \put(12.5,6){\vector(0,1){6}}
 \put(12,12){\vector(1,1){6}}
 \put(12,12){\vector(-1,1){6}}
 \put(6,18){\vector(1,1){6}}
 \put(18,18){\vector(-1,1){6}}
 \put(11.5,24){\vector(0,1){6}}
 \put(12.5,24){\vector(0,1){6}}
 \put(12,30){\vector(-1,0){12}}
 \put(12,30){\vector(0,1){7}}
 \put(-0.5,30){\vector(0,1){7}}
 \put(0.5,30){\vector(0,1){7}}
  \end{picture}}
 \put(103,-2){(d)}
   \put(128,10){\begin{picture}(20,43)  %      -  (e)
 \put(0,0){\circle{2}}
 \put(8,0){\circle{2}}
 \put(16,0){\circle{2}}
 \put(0,38){\circle{2}}
 \put(12,38){\circle{2}}
 \put(18,38){\circle{2}}
 \put(0,1){\vector(0,1){28.5}}
 \put(8,0){\vector(2,3){4}}
 \put(16,0){\vector(-2,3){4}}
 \put(12,12){\vector(1,1){6}}
 \put(12,12){\vector(-1,1){6}}
 \put(6,18){\vector(1,1){6}}
 \put(18,18){\vector(-1,1){6}}
 \put(12,30){\vector(-1,0){12}}
 \put(12,30){\vector(0,1){7}}
  \end{picture}}
 \put(135,-2){(e)}
 \end{picture}
  \end{center}
\caption{(a) $G$; \;\;(b) $\phi$; \;\;(c) $\phi'$; \;\;(d) $\xi$; \;\;(e)
$E_\phi\triangle E_{\phi'}$} \label{fig:double}
  \end{figure}

Let $\chi^{E'}$ denote the incidence vector in $\Rset^E$ of a subset $E'$ of
edges of $G$, i.e. $\chi^{E'}(e)=1$ if $e\in E'$, and 0 otherwise. A function
$\xi:E\to\{0,1,2\}$ is called a \emph{double flow} for the sets $I(A),J(A)$ as
above if there exist an $I(A)$-flow $\phi$ and a $J(A)$-flow $\phi'$ such that
$\chi^{E_\phi}+\chi^{E_{\phi'}}=\xi$. We say that $\phi,\phi'$ \emph{decompose}
$\xi$ and denote the number of such pairs $\phi,\phi'$ by
$N_{I(A),J(A)}(\xi)$.  For $i=1,2$, let $\xi^{(i)}$ denote the set of edges $e$
with $\xi(e)=i$, and let $d(\xi)$ denote the number of circuits in the subgraph
of $G$ induced by $\xi^{(1)}$ (i.e. the number $d$ in Lemma~\ref{lm:sumFFp}).
   \begin{lemma} \label{lm:dzeta}
~$N_{I(A),J(A)}(\xi)=2^{d(\xi)}$.
  \end{lemma}
  \begin{proof}
~Fix a pair $\phi\in\Phi_{I(A)}$ and $\phi'\in\Phi_{J(A)}$ decomposing $\xi$.
Let $C_1,\ldots,C_{d=d(\xi)}$ be the circuits in the subgraph induced by
$\xi^{(1)}$ (=$E_\phi\triangle E_{\phi'}$). By Lemma~\ref{lm:sumFFp}, each
circuit $C_i$ is a concatenation of (directed up to reversing) paths
$Q_1,\ldots,Q_{2k}=Q_0$, where consecutive $Q_j,Q_{j+1}$ are contained in
different flows among $\phi,\phi'$ and either both leave or both enter their
common vertex (note that $C_i$ cannot entirely belong to one of $\phi,\phi'$
since $G$ is acyclic). Therefore, exchanging the pieces $Q_j$ in $\phi,\phi'$
(i.e. replacing $E_\phi$ by $E_\phi\triangle E_{C_i}$, and $E_{\phi'}$ by
$E_{\phi'}\triangle E_{C_i}$), we obtain a decomposition of $\xi$ into another
pair of $I(A)$- and $J(A)$-flows.

The above procedure can be applied to circuits $C_i$ independently. So we can
choose an arbitrary subset $D\subseteq [d(\xi)]$. Let
$\Dscr:=\cup(E_{C_i}\colon i\in D)$. Then the edge set $E_\phi\triangle \Dscr$
induces an $I(A)$-flow $\psi$, and $E_{\phi'}\triangle \Dscr$ induces a
$J(A)$-flow $\psi'$. Furthermore, $\chi^{E_\psi}+\chi^{E_{\psi'}}=\xi$.
Obviously, different subsets $D$ produce different pairs $\psi,\psi'$
decomposing $\xi$ (e.g., the choice $D=\emptyset$ gives the initial pair
$\phi,\phi'$). Conversely, if a pair of $\psi\in\Phi_{I(A)}$ and $\psi'\in
\Phi_{J(A)}$ decomposes $\xi$, then, in view of $E_\psi\cap
E_{\psi'}=\xi^{(2)}$ and $E_\psi\triangle E_{\psi'}=\xi^{(1)}$, one can
conclude that for each circuit $C_i$, the components of $C_i\cap \psi$ and
$C_i\cap \psi'$ are the alternating subpaths $Q_j$ as above. This gives the
desired equality.
  \end{proof}

Next we are going to decompose $\xi$ into a pair of flows whose source sets are
different from $\hat S_{\gamma(A)},\hat S_{\gamma(\bar A)}$. Let
$P_1,\ldots,P_p$ be the paths as in Lemma~\ref{lm:sumFFp}. Exactly $q$ of them
connect $\hat S_{\gamma(A)}$ and $\hat S_{\gamma(\bar A)}$; we call these paths
\emph{essential} and denote their set by $\Pscr(\xi)$. For a path
$P\in\Pscr(\xi)$ with end vertices $\hat s_{\gamma(i)}$ and $\hat
s_{\gamma(j)}$, the pair $\{i,j\}$ is denoted by $\pi(P)$ (in particular,
$|\pi(P)\cap A|=|\pi(P)\cap \bar A|=1$). Define $M(\xi):=\{\pi(P)\colon
P\in\Pscr(\xi)\}$.
   \begin{lemma} \label{lm:path_switch}
Choose an arbitrary subset $\Pi\subseteq M(\xi)$. Define $Z:=\cup(\pi\in\Pi)$
and $A':=A\triangle Z$. Then $\xi$ is decomposed by exactly $2^{d(\xi)}$ pairs
formed by an $I(A')$-flow and a $J(A')$-flow. Therefore,
$N_{I(A'),J(A')}(\xi)=N_{I(A),J(A)}(\xi)$.
  \end{lemma}
  \begin{proof}
Each path $P\in \Pscr(\xi)$ is a concatenation of an even number of subpaths
$Q_1,\ldots,Q_r$ alternately contained in $\phi$ and $\phi'$. Let
$\pi(P)=\{i,j\}$; then one of $i,j$ is in $A$, and the other in $\bar A$. Also
one source among $\hat s_{\gamma(i)},\hat s_{\gamma(j)}$ belongs to $Q_1$, and
the other to $Q_r$. Exchanging in $\phi,\phi'$ the corresponding pieces $Q_k$
of $P$, we obtain a decomposition of $\xi$ into an $I(\tilde A)$-flow and a
$J(\tilde A)$-flow, where $\tilde A:=A\triangle\{i,j\}$. Now the result is
obtained by arguing as in Lemma~\ref{lm:dzeta}.
  \end{proof}

In what follows we will use the fact that, although the modified graph $G$ may
not be planar, its subgraph $G(\xi)$ induced by the edge set
$\xi^{(1)}\cup\xi^{(2)}$ is planar.

To see this, let $\xi$ be decomposed by flows $\phi,\phi'$ (as before) and
consider in the initial graph $G$ a non-terminal vertex $v$ which belongs to
both flows $\phi,\phi'$. Let $a,a'$ be the edges of $\phi$ entering and leaving
$v$, respectively, and let $b,b'$ be similar edges for $\phi'$. The only
situation when the modified graph $G$ is not locally planar in a small
neigborhood of the split-edge $e_v$ is that all $a,a',b,b'$ are different and
follow in this order (clockwise or counterclockwise) around $v$. We assert that
this is not the case. Indeed, $a,a'$ belong to a directed path $P$ in $G$ from
a source $s_i$ to a sink $t_{i'}$, and similarly there is a directed path $Q$
from $s_j$ to $t_{j'}$ containing $b,b'$. From the facts that the initial graph
$G$ is planar and acyclic and that the edges $a,a',b,b'$ occur in this order
around $v$ one can conclude that the paths $P,Q$ can meet only at $v$. This
implies that the terminals $s_i,t_{i'},s_j,t_{j'}$ are different and follow in
this order in the boundary of $G$, yielding a contradiction. Thus, $G(\xi)$ is
planar, as required.

%----------------------- Sec.3

\section{\Large Balanced collections and the main theorem}  \label{sec:balan}

In this section we use the above observations and results to construct
collections providing stable quadratic relations.

As before, consider a set $A\in\binom{[p+q]}{p}$, a double flow $\xi$ for
$I(A),J(A)$ and the set of pairs $M=M(\xi)$. It will be convenient to think
that all pairs are ordered: if a pair $\pi$ consists of elements $i,j$ and
$i<j$, we write $\pi=(i,j)$ or $\pi=ij$ and call it an \emph{arc} in $M$. We
denote the interval $\{i,i+1,\ldots,j\}$ by $[i..j]$ or by $[\pi]$ and say that
an element $k$ is \emph{covered} by $\pi$ if $k\in [\pi]$. An element in
$[p+q]-\cup(\pi\in M)$ is called \emph{free}.

We observe that $M$ possesses the following properties:
  \begin{numitem1}
  \begin{itemize}
\item[(i)] $|M|=q$, the arcs in $M$ are mutually disjoint, and $|\pi\cap A|=
|\pi\cap\bar A|=1$ for each $\pi\in M$;
\item[(ii)] the set $M$ is \emph{nested}, which means that for any two arcs
$\pi,\pi'\in M$, the intervals $[\pi]$ and $[\pi']$ are either disjoint or one
includes the other;
\item[(iii)] no free element is covered by an arc in $M$ (i.e. $\pi\in M$ and
$k\in[\pi]$ imply $k\in\pi'$ for some $\pi'\in M$).
  \end{itemize}
   \label{eq:nest}
   \end{numitem1}

\noindent Indeed, (i) is obvious. Violation of~(ii) means the existence of arcs
$\pi=ij$ and $\pi'=i'j'$ in $M$  such that $i<i'<j<j'$. Then the sources $
s_{\gamma(i)},s_{\gamma(i')},s_{\gamma(j)},s_{\gamma(j')}$ follow in this order
in the boundary of the initial $G$. Since the graph $G(\xi)$ is planar, the
path $P\in\Pscr(\xi)$ connecting $\hat s_{\gamma(i)},\hat s_{\gamma(j)}$
intersects the path $P'\in\Pscr(\xi)$ connecting $\hat s_{\gamma(i')},\hat
s_{\gamma(j')}$. But $P,P'$ must be disjoint (cf. Lemma~\ref{lm:sumFFp}). To
see~(iii), consider an arc $\pi=ij\in M$ and a free element $k$. Since $k$ is
free, the subgraph induced by $\xi^{(1)}$ contains a path $P$ connecting the
source $\hat s_{\gamma(k)}$ and some sink $\hat t_r\in\tilde T$. In case
$k\in[\pi]$, the path $P$ would intersect the path in $\Pscr(\xi)$ connecting
$\hat s_{\gamma(i)}$ and $\hat s_{\gamma(j)}$ (since $G(\xi)$ is planar and
$\hat s_{\gamma(i)},\hat s_{\gamma(k)},\hat s_{\gamma(j)},\hat t_r$ follow in
this order in its boundary), which is impossible.
\smallskip

The above observations inspire consideration of more abstract objects. A set
$M$ of ordered pairs (arcs) in $[p+q]$ satisfying~\refeq{nest} is called a
\emph{feasible matching} for $A$. The set of all feasible matchings for $A$ is
denoted by $\Mscr(A)$, and we refer to a pair $(A,M)$, where $M\in\Mscr(A)$, as
a \emph{configuration}. For a collection $\Ascr\subseteq\binom{[p+q]}{p}$, the
(multi)set of all configurations $(A,M)$ with $A\in\Ascr$ is denoted by
$\Kscr(\Ascr)$.

The \emph{exchange operation} applied to a configuration $(A,M)$ and to a
chosen subset $\Pi\subseteq M$ makes the $p$-element set
$A':=A\triangle(\cup(\pi\in\Pi))$; in other words, we swap the elements of $A$
and $\bar A$ in each arc $\pi\in\Pi$. Clearly $M$ becomes a feasible matching
for $A'$, and the exchange operation applied to the configuration $(A',M)$ and
the same subset $\Pi$ returns $A$.
\medskip

 \noindent\textbf{Definition.}
Let us say that two (multi)collections $\Ascr,\Ascr'\subseteq\binom{[p+q]}{p}$
are \emph{balanced} if there exists a bijection of $\Kscr(\Ascr)$ to
$\Kscr(\Ascr')$ that sends each configuration $(A,M)$ in the former to a
configuration $(A',M)$ in the latter. (We rely on the simple fact that if
$(A,M)$ and $(A',M)$ are two configurations with the same matching $M$, then
$A'$ can be obtained from $A$ by the exchange operation w.r.t. some
$\Pi\subseteq M$.) Equivalently, $\Ascr,\Ascr'$ are balanced if for each
matching $M$ in $[p+q]$, the number of times $M$ occurs in sets $\Mscr(A)$
among $A\in\Ascr$ is equal to a similar number in sets $\Mscr(A')$ among
$A'\in\Ascr'$. We can express this condition as
  $$
  \Mscr(\Ascr)=\Mscr(\Ascr'),
  $$
where for a collection $\Ascr''\subseteq\binom{[p+q]}{p}$, ~$\Mscr(\Ascr'')$
denotes the multiset consisting of matchings $M$ taken with multiplicities
$|\{A\in\Ascr''\colon M\in\Mscr(A)\}|$.
\medskip

This notion plays a central role in our description, and the main result is as
follows.
  \begin{theorem} \label{tm:main}
~Let $\Ascr,\Ascr'\subseteq\binom{[p+q]}{p}$. The following statements are
equivalent:

{\rm (i)} ~\refeq{S_pluck} is a stable quadratic relation;

{\rm (ii)} the pair $\Ascr,\Ascr'$ is balanced.
  \end{theorem}

Part (i)$\Rightarrow$(ii) of this theorem will be shown in
Section~\ref{sec:(i)-(ii)}. In its turn, part (ii)$\Rightarrow$(i) can be
immediately proved by relying on the lemmas from the previous section.
  \begin{prop} \label{pr:balance-vq}
~Let $\Ascr,\Ascr'\subseteq\binom{[p+q]}{p}$ be balanced. Then~\refeq{S_pluck}
holds for any disjoint subsets $X,Y\subseteq[n]$ with $|Y|=p+q$ and any
SFG-function $f$ on $2^{[n]}$ (concerning arbitrary $G,w,\mathfrak{S}$ as
above).
  \end{prop}
  \begin{proof}
~Fix corresponding $G,w,\mathfrak{S},X,Y$ and consider the function $f=f_w$
determined by the weighting $w$. For $A\in\binom{[p+q]}{p}$, let $\Xi(A)$
denote the set of all distinct double flows $\xi$ for $I(A),J(A)$ (considering
$G$ in the modified form). The summand concerning $A\in\Ascr$ in the l.h.s.
of~\refeq{S_pluck} can be expressed via double flows as follows:
  \begin{multline} \label{eq:ff-zeta}
  f(I(A))\odot f(J(A))=\left(\bigoplus\nolimits_{\phi\in\Phi_{I(A)}} w(\phi)\right)
     \odot \left(\bigoplus\nolimits_{\phi'\in\Phi_{J(A)}} w(\phi')\right) \\
  =\bigoplus\nolimits_{\xi\in\Xi(A)} N_{I(A),J(A)}(\xi)
  \; w^{\odot\xi},\quad
  \end{multline}
where $N_{I(A),J(A)}(\xi)$ is the number of pairs
$(\phi,\phi')\in(\Phi_{I(A)},\Phi_{J(A)})$ with
$\chi^{E_\phi}+\chi^{E_{\phi'}}=\xi$ (cf. Lemma~\ref{lm:dzeta}), and
$w^{\odot\xi}$ is the function on the set $E_V$ of split-edges taking values
$w(e)\odot\ldots\odot w(e)$ ($\xi(e)$ times), $e\in E_V$. Do similarly for the
summand concerning $A'\in\Ascr'$ in the r.h.s. of~\refeq{S_pluck}.

We associate to each $\xi\in\Xi(A)$ the configuration $K_{A,\xi}:=(A,M(\xi))$.
Let $\beta:\Kscr(\Ascr)\to\Kscr(\Ascr')$ be the corresponding bijection
(existing as $\Ascr,\Ascr'$ are balanced). For $A\in\Ascr$ and $\xi\in\Xi(A)$,
~$\beta$ sends the configuration $K_{A,\xi}$ to a configuration $(A',M)$ with
$A\in\Ascr'$ and $M=M(\xi)$. Then $\xi$ is a double flow for $I(A'),J(A')$ as
well. Therefore, $\beta$ gives a 1--1 correspondence between the set of pairs
$(A\in\Ascr,\xi\in\Xi(A))$ and the set of pairs $(A'\in\Ascr',\xi'\in\Xi(A'))$.
Moreover, corresponding pairs $(A,\xi)$ and $(A',\xi')$ are such that
$\xi=\xi'$. By Lemma~\ref{lm:path_switch}, we have $N_{I(A),J(A)}(\xi)
=N_{I(A'),J(A')}(\xi)$. Now the desired equality~\refeq{S_pluck} follows by
comparing the $\oplus$-sum of the last terms in~\refeq{ff-zeta} over
$A\in\Ascr$ with a similar sum over $A'\in\Ascr'$.
  \end{proof}

%----------------------- Sec. 4

\section{\Large Examples of stable quadratic relations}  \label{sec:relat}

In this section we illustrate the method described in the previous section by
exhibiting several classes of stable quadratic relations on SFG-functions.
According to Proposition~\ref{pr:balance-vq}, once we are able to show that one
or another pair of collections $\Ascr,\Ascr'\subseteq \binom{[p+q]}{p}$ is
balanced, we can declare that relation~\refeq{S_pluck} involving these
collections is stable. Recall that speaking of a stable quadratic relation
(with collections $\Ascr,\Ascr'$ fixed), or an \emph{sq-relation} for short, we
mean that~\refeq{S_pluck} holds for any SFG-function $f:2^{[n]}\to
\mathfrak{S}$ (concerning arbitrary $G,w,\mathfrak{S}$) and any disjoint sets
$X,Y\subseteq[n]$ with $|Y|=p+q$.

When considering and visualizing one or another ordered partition $(A,\bar A)$
of $[p+q]$, it will be convenient for us to call elements of $A$ \emph{white},
and elements of $\bar A$ \emph{black}. \medskip

\noindent \textbf{1.} When $p=2$ and $q=1$, the collection $\binom{[p+q]}{p}$
consists of three 2-element sets $A$, namely, 12,13,23, and their complements
$\bar A$ are the 1-element sets 3,2,1, respectively. Since $q=1$, a feasible
matching consists of a unique arc. The sets 12 and 23 admit only one feasible
matching each, namely, $\Mscr(12)=\{\{23\}\}$ and $\Mscr(23)=\{\{12\}\}$,
whereas 13 has two feasible matchings, namely, $\Mscr(13)=\{\{12\},\{23\}\}$.
Therefore, the collections $\Ascr:=\{13\}$ and $\Ascr':=\{12,23\}$ are
balanced. The corresponding configurations and bijection are illustrated in the
picture where the 2-element sets (forming $\Ascr,\Ascr'$) and their 1-element
complements are indicated by white and black circles, respectively.

 \begin{center}
  \unitlength=1mm
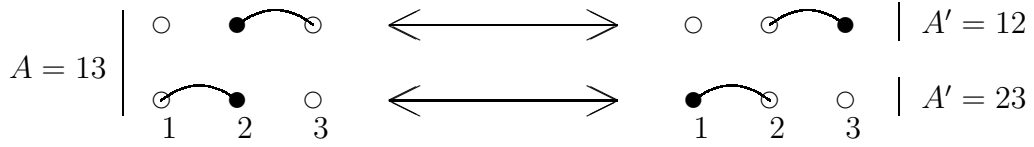
  \begin{picture}(140,20)
  \put(20,5){\circle{2}}
  \put(20,15){\circle{2}}
  \put(40,5){\circle{2}}
  \put(40,15){\circle{2}}
  \put(90,15){\circle{2}}
  \put(100,5){\circle{2}}
  \put(100,15){\circle{2}}
  \put(110,5){\circle{2}}
  \put(30,5){\circle*{2}}
  \put(30,15){\circle*{2}}
  \put(90,5){\circle*{2}}
  \put(110,15){\circle*{2}}
  \multiput(50,5)(0,10){2}%
{\put(0,0){\line(1,0){30}}
  \put(0,0){\line(2,1){4}}
  \put(0,0){\line(2,-1){4}}
  \put(30,0){\line(-2,1){4}}
  \put(30,0){\line(-2,-1){4}}}
  \multiput(20,5)(0,0){1}%
{\qbezier(0,0)(5,4)(10,0)}
  \multiput(30,15)(0,0){1}%
{\qbezier(0,0)(5,4)(10,0)}
  \multiput(90,5)(0,0){1}%
{\qbezier(0,0)(5,4)(10,0)}
  \multiput(100,15)(0,0){1}%
{\qbezier(0,0)(5,4)(10,0)}
  \put(20,0){1}
  \put(30,0){2}
  \put(40,0){3}
  \put(90,0){1}
  \put(100,0){2}
  \put(110,0){3}
  \put(15,3){\line(0,1){14}}
  \put(0,8){$A=13$}
  \put(117,3){\line(0,1){5}}
  \put(120,4){$A'=23$}
  \put(117,13){\line(0,1){5}}
  \put(120,14){$A'=12$}
  \end{picture}
   \end{center}

This gives rise to an sq-relation on triples in $[n]$ (generalizing AP3- and
TP3-relations~\refeq{AP3},\refeq{TP3}): for any $i<j<k$ (forming $Y$) and
$X\subseteq[n]-\{i,j,k\}$, one holds:
   \begin{equation} \label{eq:SP3}
    f(Xik)\odot f(Xj)=(f(Xij)\odot f(Xk))\oplus(f(Xjk)\odot f(Xi)).
    \end{equation}

 \medskip
\noindent \textbf{2.} Let $p=q=2$. Take the collections $\Ascr:=\{13\}$ and
$\Ascr':=\{12,14\}$ in $\binom{[4]}{2}$. One can see that each of 12 and 14
admits a unique feasible matching: $\Mscr(12)=\{14,23\}$ and
$\Mscr(14)=\{12,34\}$, whereas $\Mscr(13)$ consists of two feasible matchings,
just the same $\{14,23\}$ and $\{12,34\}$. Therefore, $\Ascr,\Ascr'$ are
balanced; see the picture where the arcs involved in the corresponding exchange
operations are marked with crosses.

 \begin{center}
  \unitlength=1mm
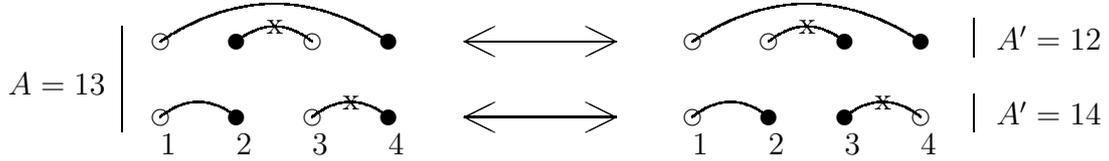
  \begin{picture}(150,20)
  \put(20,5){\circle{2}}
  \put(20,15){\circle{2}}
  \put(40,5){\circle{2}}
  \put(40,15){\circle{2}}
  \put(90,15){\circle{2}}
  \put(90,5){\circle{2}}
  \put(100,15){\circle{2}}
  \put(120,5){\circle{2}}
  \put(30,5){\circle*{2}}
  \put(30,15){\circle*{2}}
  \put(50,5){\circle*{2}}
  \put(50,15){\circle*{2}}
  \put(100,5){\circle*{2}}
  \put(110,5){\circle*{2}}
  \put(110,15){\circle*{2}}
  \put(120,15){\circle*{2}}
  \multiput(60,5)(0,10){2}%
{\put(0,0){\line(1,0){20}}
  \put(0,0){\line(2,1){4}}
  \put(0,0){\line(2,-1){4}}
  \put(20,0){\line(-2,1){4}}
  \put(20,0){\line(-2,-1){4}}}
  \multiput(20,5)(0,0){1}%
{\qbezier(0,0)(5,4)(10,0)}
  \multiput(40,5)(0,0){1}%
{\qbezier(0,0)(5,4)(10,0)}
  \multiput(30,15)(0,0){1}%
{\qbezier(0,0)(5,4)(10,0)}
  \multiput(90,5)(0,0){1}%
{\qbezier(0,0)(5,4)(10,0)}
  \multiput(110,5)(0,0){1}%
{\qbezier(0,0)(5,4)(10,0)}
  \multiput(100,15)(0,0){1}%
{\qbezier(0,0)(5,4)(10,0)}
  \multiput(20,15)(0,0){1}%
{\qbezier(0,0)(15,10)(30,0)}
  \multiput(90,15)(0,0){1}%
{\qbezier(0,0)(15,10)(30,0)}
  \put(44,6){x}
  \put(34,16){x}
  \put(114,6){x}
  \put(104,16){x}
  \put(20,0){1}
  \put(30,0){2}
  \put(40,0){3}
  \put(50,0){4}
  \put(90,0){1}
  \put(100,0){2}
  \put(110,0){3}
  \put(120,0){4}
  \put(15,3){\line(0,1){14}}
  \put(0,8){$A=13$}
  \put(127,3){\line(0,1){5}}
  \put(130,4){$A'=14$}
  \put(127,13){\line(0,1){5}}
  \put(130,14){$A'=12$}
  \end{picture}
   \end{center}

As a consequence, we obtain an sq-relation on quadruples
(generalizing~\refeq{AP4} and its tropical counterpart): for any $i<j<k<\ell$
(forming $Y$) and $X\subseteq[n]-\{i,j,k,\ell\}$, one holds:
   \begin{equation} \label{eq:SP4}
    f(Xik)\odot f(Xj\ell)=(f(Xij)\odot f(Xk\ell))\oplus(f(Xi\ell)\odot f(Xjk)).
    \end{equation}

 \medskip
\noindent \textbf{3.} As one more illustration of the method, let us consider
one particular case for $p=3$ and $q=2$. Put $\Ascr:=\{135\}$ and
$\Ascr':=\{234,125,145\}$. One can check that $\Mscr(234)=\{\{12,45\}\}$,
~$\Mscr(125):=\{\{14,23\},\{23,45\}\}$, ~$\Mscr(145)=\{\{12,34\},\{25,34\}\}$,
and that $\Mscr(135)$ consists just of the five matchings occurring in those
three collections. Therefore, $\Ascr,\Ascr'$ are balanced. The corresponding
configurations and bijection are shown in the picture.

 \begin{center}
  \unitlength=1mm
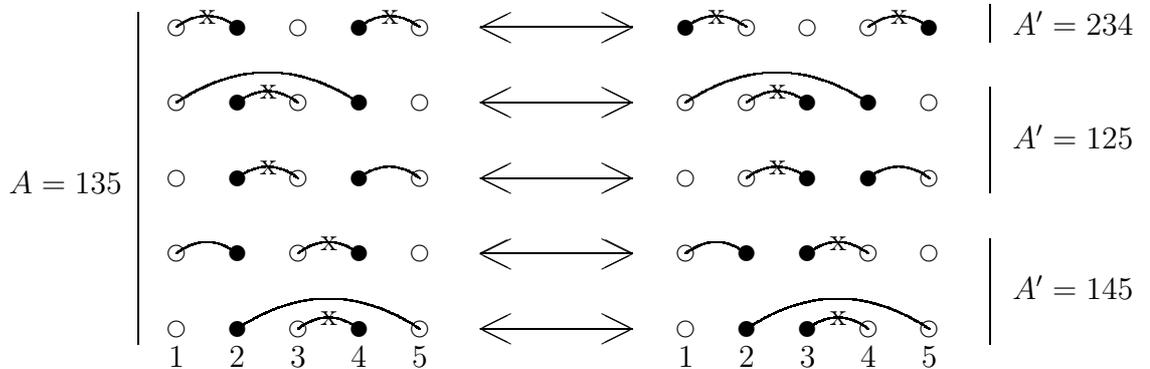
  \begin{picture}(140,47)
\multiput(20,5)(0,10){5}%
  {\put(0,0){\circle{2}}
  \put(8,0){\circle*{2}}
  \put(16,0){\circle{2}}
  \put(24,0){\circle*{2}}
  \put(32,0){\circle{2}}}
\multiput(87,5)(0,10){2}%
  {\put(0,0){\circle{2}}
  \put(8,0){\circle*{2}}
  \put(16,0){\circle*{2}}
  \put(24,0){\circle{2}}
  \put(32,0){\circle{2}}}
\multiput(87,25)(0,10){2}%
  {\put(0,0){\circle{2}}
  \put(8,0){\circle{2}}
  \put(16,0){\circle*{2}}
  \put(24,0){\circle*{2}}
  \put(32,0){\circle{2}}}
\multiput(87,45)(0,10){1}%
  {\put(0,0){\circle*{2}}
  \put(8,0){\circle{2}}
  \put(16,0){\circle{2}}
  \put(24,0){\circle{2}}
  \put(32,0){\circle*{2}}}
  \multiput(60,5)(0,10){5}%
{\put(0,0){\line(1,0){20}}
  \put(0,0){\line(2,1){4}}
  \put(0,0){\line(2,-1){4}}
  \put(20,0){\line(-2,1){4}}
  \put(20,0){\line(-2,-1){4}}}
  \multiput(36,5)(67,0){2}%
{\qbezier(0,0)(4,3)(8,0)
  \put(3,0.5){x}}
  \multiput(20,15)(67,0){2}%
{\qbezier(0,0)(4,3)(8,0)}
  \multiput(36,15)(67,0){2}%
{\qbezier(0,0)(4,3)(8,0)
  \put(3,0.5){x}}
  \multiput(28,25)(67,0){2}%
{\qbezier(0,0)(4,3)(8,0)
  \put(3,0.5){x}}
  \multiput(44,25)(67,0){2}%
{\qbezier(0,0)(4,3)(8,0)}
  \multiput(28,35)(67,0){2}%
{\qbezier(0,0)(4,3)(8,0)
  \put(3,0.5){x}}
  \multiput(20,45)(67,0){2}%
{\qbezier(0,0)(4,3)(8,0)
  \put(3,0.5){x}}
  \multiput(44,45)(67,0){2}%
{\qbezier(0,0)(4,3)(8,0)
  \put(3,0.5){x}}
  \multiput(28,5)(67,0){2}%
{\qbezier(0,0)(12,8)(24,0)}
  \multiput(20,35)(67,0){2}%
{\qbezier(0,0)(12,8)(24,0)}
  \multiput(-1,0)(67,0){2}%
  {\put(20,0){1}
  \put(28,0){2}
  \put(36,0){3}
  \put(44,0){4}
  \put(52,0){5}}
  \put(15,3){\line(0,1){44}}
  \put(-2,23){$A=135$}
  \put(127,3){\line(0,1){14}}
  \put(130,9){$A'=145$}
  \put(127,23){\line(0,1){14}}
  \put(130,29){$A'=125$}
  \put(127,43){\line(0,1){5}}
  \put(130,44){$A'=234$}
  \end{picture}
   \end{center}

This implies a particular sq-relation on quintuples: for $i<j<k<\ell<m$ and
$X\subseteq [n]-\{i,j,k,\ell,m\}$, one holds:
   \begin{multline} \label{eq:SP5}
    f(Xikm)\odot f(Xj\ell)
    =(f(Xjk\ell)\odot f(Xim))\oplus(f(Xijm)\odot f(Xk\ell))\\
      \oplus(f(Xi\ell m)\odot f(Xjk)).\quad
    \end{multline}

 \medskip
 \noindent \textbf{4.}
Next we describe a wide class of balanced collections for arbitrary $p\ge q$;
it includes the collections indicated in items 1 and 2 as very special cases.

The collection $\Ascr$ that we are going to construct contains two
distinguished sets $B_0,B_1$. The set $B_0$ is the interval $[p]$. Then $\bar
B_0=[p+1..p+q]$ and $\Mscr(B_0)$ consists of a unique matching $M_0$: its arcs
are the pairs $\pi_i:=(p-i+1,p+i)$ for $i=1,\ldots,q$. The set $B_1$ is
obtained by applying to $B_0$ the exchange operation w.r.t. a chosen nonempty
subset $\Pi_0\subseteq M_0$, i.e. $B_1=L\cup R$, where
   \begin{equation} \label{eq:LR}
  L:=[p-q]\cup\{p-i+1\colon \pi_i\not\in\Pi_0\}\quad \mbox{and} \quad
   R:=\{p+i\colon \pi_i\in\Pi_0\}.
   \end{equation}
An example for $p=5,q=4$ is drawn in the picture where the arcs in $\Pi_0$ are
marked with crosses.

 \begin{center}
  \unitlength=1mm
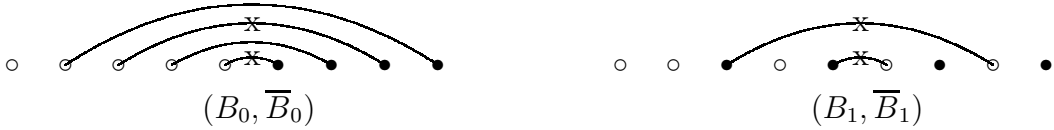
  \begin{picture}(140,22)
                         %              left
   \put(0,3){\begin{picture}(57,20)
  \put(0,5){\circle{1.5}}
  \put(7,5){\circle{1.5}}
  \put(14,5){\circle{1.5}}
  \put(21,5){\circle{1.5}}
  \put(28,5){\circle{1.5}}
  \put(35,5){\circle*{1.5}}
  \put(42,5){\circle*{1.5}}
  \put(49,5){\circle*{1.5}}
  \put(56,5){\circle*{1.5}}
  \qbezier(7,5)(31.5,21)(56,5)
  \qbezier(14,5)(31.5,16)(49,5)
  \put(30.5,9.5){x}
  \qbezier(21,5)(31.5,11)(42,5)
  \qbezier(28,5)(31.5,7)(35,5)
  \put(30.5,5){x}
  \put(25,-2){$(B_0,\bar B_0)$}
  \end{picture}}
                         %              right
   \put(80,3){\begin{picture}(56,20)
  \put(0,5){\circle{1.5}}
  \put(7,5){\circle{1.5}}
  \put(14,5){\circle*{1.5}}
  \put(21,5){\circle{1.5}}
  \put(28,5){\circle*{1.5}}
  \put(35,5){\circle{1.5}}
  \put(42,5){\circle*{1.5}}
  \put(49,5){\circle{1.5}}
  \put(56,5){\circle*{1.5}}
  \qbezier(14,5)(31.5,16)(49,5)
  \put(30.5,9.5){x}
  \qbezier(28,5)(31.5,7)(35,5)
  \put(30.5,5){x}
  \put(25,-2){$(B_1,\bar B_1)$}
  \end{picture}}
  \end{picture}
   \end{center}

The other members of $\Ascr\cup\Ascr'$ have the same tail part $R$ as the set
$B_1$. More precisely, take the collection
   $$
%   \Bscr:=\{A\in\binom{[p+q]}{p}\colon A\cap[p+1..p+q]=R\},
   \Bscr:=\{A\subseteq[p+q] \colon |A|=p,\; A\cap[p+1..p+q]=R\}.
   $$
For a subset $A\subseteq[n]$, define $\Sigma(A):=\sum(i\in A)$. Now put
   \begin{gather}
   \Ascr:=\{B_0\}\cup\{ A\in\Bscr\colon\Sigma(A)-\Sigma(B_1)\;\; \mbox{odd}\}
   \quad\mbox{and} \nonumber \\
   \Ascr':=\{ A\in\Bscr\colon\Sigma(A)-\Sigma(B_1)\;\; \mbox{even}\}.
   \label{eq:AA4}
   \end{gather}
In particular, $B_0\in\Ascr$, $B_1\in\Ascr'$, and $\Ascr\cap\Ascr'=\emptyset$.
   \begin{lemma} \label{lm:AA4}
The pair $\Ascr,\Ascr'$ as in~\refeq{AA4} is balanced.
  \end{lemma}
  \begin{proof}
~Consider a set $A\in\Ascr\cup\Ascr'$ and a matching $M\in\Mscr(A)$. We
describe a rule which associates to $(A,M)$ another configuration $(A',M)$
(aiming to obtain a bijection between $\Kscr(\Ascr)$ and $\Kscr(\Ascr')$). Two
cases are possible. \smallskip

\noindent\emph{Case 1}: $A=B_0$. Then $M=M_0$, and the configuration
$(B_0,M_0)$ belongs to $\Kscr(\Ascr)$. We naturally associate to $(B_0,M_0)$
the configuration $(B_1,M_0)$ in $\Kscr(\Ascr')$. (Note that $B_0$ and $B_1$
are linked by the exchange operation w.r.t. the set $\Pi_0\subseteq M_0$.)
\smallskip

\noindent\emph{Case 2}: $A\ne B_0$. Then $A\in\Bscr$. Suppose $M$ consists of
arcs $\rho_k=i_kj_k$, $k=1,\ldots,q$, and let $j_1<j_2<\ldots<j_q$. Let us say
that an arc $\rho_k$ is \emph{short} if $j_k=i_k+1$. It is immediate
from~\refeq{nest} that the interval of any arc $\rho_{k'}$ contains a short arc
$\rho_k$, i.e. $i_{k'}\le i_k<j_k\le j_{k'}$. This implies that the arc
$\rho_1$ is short, in view of $j_1<\ldots<j_q$. Consider two
subcases.\smallskip

\noindent\emph{Subcase 2a}: $j_1\ge p+1$. This is possible only if $j_k=p+k$
for all $k=1,\ldots,q$. Then $[\rho_1]\subset[\rho_2]\subset\ldots
\subset[\rho_q]$, implying $i_k=p+1-k$ (in view of~\refeq{nest}(ii),(iii)).
Therefore, $M$ coincides with $M_0$, and now the condition $A\cap[p+1..p+q]=R$
implies that $A$ coincides with the set $B_1$. So $(A,M)$ is the configuration
$(B_1,M_0)$ in $\Kscr(\Ascr')$, and we associate to it the configuration
$(B_0,M_0)$, to be agreeable with Case~1. \smallskip

\noindent\emph{Subcase 2b}: $j_1\le p$. We associate to $(A,M)$ the
configuration $(A',M)$ with $A':=A\triangle\rho_1$ (i.e. we apply to $(A,M)$
the exchange operation w.r.t. the subset of $M$ formed by the singleton
$\{\rho_1\}$). Obviously, $A'\cap[p+1..p+q]=R$, whence $A'\in\Bscr$. Also
$j_1-i_1=1$ implies that $\Sigma(A)-\Sigma(A')$ is odd. Thus, one of
$(A,M),(A',M)$ belongs to $\Kscr(\Ascr)$ and the other to $\Kscr(\Ascr')$. We
associate these configurations to each other (taking into account that the same
short arc $\rho_1$ in $M$ is chosen in both cases).
  \end{proof}

\noindent \textbf{Remark 2.} ~(i) When $p=2$ and $q=1$, we have $B_0=12$, $\bar
B_0=3$ and $M_0=\{23\}$. Taking $\Pi_0=M_0$, we obtain $B_1=13$ and
$\Bscr=\{13,23\}$. This gives $\Ascr=\{12,23\}$ and $\Ascr'=\{13\}$, which
matches the balanced collections described in item~1. (ii) When $p=q=2$, we
have $B_0=12$, $\bar B_0=34$ and $M_0=\{23,14\}$. Taking $\Pi_0=\{23\}$, we
obtain $B_1=13$ and $\Bscr=\{13,23\}$. This gives $\Ascr=\{12,23\}$ and
$\Ascr'=\{13\}$, which is equivalent to the balanced collections $\{13\}$,
$\{12,14\}$ in item~2. \medskip

Pairs $\Ascr,\Ascr'$ as in~\refeq{AA4} give rise to sq-relations on
SFG-functions which are viewed as follows. Let $Y\subseteq[n]$ consist of
elements $i_1<\ldots<i_p<j_1<\ldots<j_q$ and let $X\subseteq[n]-Y$. Put $I:=
\{i_1,\ldots,i_p\}$ and $J:=\{j_1,\ldots,j_q\}$ and choose a subset $R\subseteq
J$ (which corresponds to $R$ in~\refeq{LR}). Then the corresponding sq-relation
is:
  \begin{multline} \label{eq:Gr_pluck}
  (f(X\cup I)\odot f(X\cup J))\oplus \bigoplus\nolimits_{I'\in\Iscr^{\rm odd}}
  f(X\cup (I-I')\cup R)    \odot f(X\cup I'\cup(J-R)     \\
  =\bigoplus\nolimits_{I''\in\Iscr^{\rm even}} f(X\cup (I-I'')\cup R)
    \odot f(X\cup I''\cup(J-R).
  \end{multline}
Here the collection $\Iscr^{\rm even}$ (resp. $\Iscr^{\rm odd}$) is formed by
the subsets $\tilde I\subseteq I$ such that $|\tilde I|=|R|$ and the integers
$\sum(k\colon j_k\in R)$ and $\sum(p+1-k\colon i_k\in\tilde I)$ have the same
(resp. different) parity.

When $p=q$ and $\mathfrak{S}=\mathbb{C}$, relations similar to~\refeq{Gr_pluck}
appear in a characterization of the Grassmannian $G_{d,n}$. In this case one
should take all possible tuples $p,Y,X,R$ such that $2\le p\le d$, $|Y|=2p$,
$|X|=d-p$ and $1\le|R|\le p$. Then the corresponding counterparts
of~\refeq{Gr_pluck} involving such tuples give a basis for the homogeneous
co-ordinate ring of $G_{d,n}$ related to the Pl\"ucker embedding of $G_{d,n}$
into $\mathbb{P}(\wedge^d \mathbb{C}^n)$; cf.~\cite{Fu}.

 \medskip
 \noindent \textbf{5.}
One more representable class of balanced collections for arbitrary $p\ge q$ is
obtained by slightly modifying the previous construction.

Fix a subset $Q\subseteq [p+2..p+q]$ and form the collection $\Cscr:=\{A\in
\binom{[p+q]}{p}\colon A\cap[p+2..p+q]=Q\}$. We partition $\Cscr$ into two
subcollections
   \begin{equation} \label{eq:AA5}
\Ascr:=\{A\in\Cscr\colon \Sigma(A)\;\;\mbox{odd}\}\quad \mbox{and}\quad
 \Ascr':=\{A\in\Cscr\colon \Sigma(A)\;\;\mbox{even}\}.
   \end{equation}

 \begin{lemma} \label{lm:AA5}
The pair $\Ascr,\Ascr'$ as in~\refeq{AA5} is balanced.
  \end{lemma}
  \begin{proof}
~Let $A\in\Cscr$ and $M\in\Mscr(A)$. Take the short arc $(i,i+1)\in M$ with $i$
minimum. We assert that $i\le p$. For otherwise any arc $i'j'\in M$ would
satisfy $j'>p+1$ (by the argument as in the proof of Lemma~\ref{lm:AA4}), which
is impossible since $|M|=q$ and $(p+q)-(p+1)<q$.

Thus, the set $A':=A\triangle\{i,i+1\}$ belongs to $\Cscr$ as well.
Furthermore, $A,A'$ belong to different collections among $\Ascr,\Ascr'$.
Associating such $(A,M),(A',M)$ to each other, we obtain the desired bijection
between $\Kscr(\Ascr)$ and $\Kscr(\Ascr')$.
  \end{proof}

This lemma gives rise to the corresponding class of sq-relations; we omit it
here.

 \medskip
 \noindent \textbf{6.}
Our last illustration to the method concerns sq-relations analogous to ones
yielding a Gr\"obner basis for the ideal ${\rm ker}(\psi)$ mentioned in the
Introduction (cf.~\cite[Sec.~14.2]{MS}).

We identify a subset of $[n]$ with the sequence of its elements in the
increasing order and consider the known partial order on $2^{[n]}$ in which for
subsets $A=(a_1<\ldots<a_p)$ and $B=(b_1<\ldots<b_q)$, one puts $A\prec B$ if
$p\ge q$ and $a_i\le b_i$ for $i=1,\ldots,q$.

For $p,q$ as before (i.e. $p\ge q$ and $p+q\le n$), take a set
$B=(b_1<\ldots<b_p)\in\binom{[p+q]}{p}$ incomparable with its complement $\bar
B=(\bar b_1< \ldots<\bar b_q)$. Then there is $d\le q$ such that $b_d>\bar b_d$
(usually one takes the smallest $d$ with this property, but this is not
important for us). Using this $d$, we partition each of $B,\bar B$ into two
subsets (where $\Bleft$ or ${\bar B}^{\rm right}$ may be empty):
  $$
  \Bleft:=\{b_1,\ldots,b_{d-1}\}, \qquad \Bright:=\{b_d,\ldots,b_p\},
  $$
  $$
 {\bar B}^{\rm left}:=\{\bar b_1,\ldots,\bar b_{d}\}, \qquad
   {\bar B}^{\rm right}:=\{b_{d+1},\ldots,\bar b_p\},
   $$
and form the set
  $$
  C:={\bar B}^{\rm left} \cup \Bright.
  $$

By the choice of $d$, $C$ begins with $d$ black elements and ends with $p-d+1$
white elements (thinking of elements of $B$ and $\bar B$ as white and black,
respectively). Introduce the following collection of $p$-element subsets of
$[p+q]$:
  $$
  \Bscr:=\{\Bleft\cup Z\colon~ Z\subset C,~ |Z|=p-d+1\}.
  $$

Then the complement $\bar A$ of any member $A=\Bleft\cup Z$ of $\Bscr$ is the
$q$-element set $(C-Z)\cup{\bar B}^{\rm right}$. In other words, for each
$A\in\Bscr$, the pair $(A,\bar A)$ is obtained from $(B,\bar B)$ by swapping
$\delta\le \min\{d,p-d+1\}$ black and white elements in the parts ${\bar
B}^{\rm left}$ and $\Bright$ of $C$, respectively. We partition $\Bscr$ into
two collections:
  \begin{equation} \label{eq:grebn}
  \Ascr:=\{A\in\Bscr\colon~ \Sigma(A)~\mbox{odd}\}\quad \mbox\quad
 \Ascr':=\{A\in\Bscr\colon~ \Sigma(A)~\mbox{even}\}.
   \end{equation}
In particular, the set $B$ belongs to one of these collections (and is the
greatest element in the poset $(\Bscr,\prec)$).

 \begin{lemma} \label{lm:AA6}
The pair $\Ascr,\Ascr'$ as in~\refeq{grebn} is balanced.
  \end{lemma}
  \begin{proof}
~Consider $A\in\Bscr$ and $M\in\Mscr(B)$. The set $C$ contains exactly $d$
elements of $\bar A$, and the set $[p+q]-C$ contains $d-1$ elements of $A$
(those forming $\Bleft$). Also $|M|=|\bar A|=q$.

Therefore, $M$ contains at least one arc $(i,j)$ with both ends $i,j$ in $C$.
We take such an arc $(i,j)$ with $j$ minimum and associate to $(A,M)$ the
configuration $(A',M)$ with $A':=A\triangle \{i,j\}$. Then $A'\in\Bscr$.
Moreover, the fact that $j-i$ is odd (since the interval $[i..j]$ is
partitioned into arcs in $M$) implies that $\Sigma(A)-\Sigma(A')$ is odd,
whence $A$ and $A'$ belong to different collections among $\Ascr,\Ascr'$.
Finally, by the canonical choice of $(i,j)$, the configuration associated to
$(B',M)$ is just $(B,M)$.
  \end{proof}

%----------------------- Sec.5

\section{\Large Necessity of the balancedness}  \label{sec:(i)-(ii)}

In this section we prove the other direction in Theorem~\ref{tm:main}. In fact,
a sharper property takes place, saying that for non-balanced $\Ascr,\Ascr'$,
the corresponding quadratic relation is not valid in a very special case.

  \begin{prop} \label{eq:sq-balance}
Let $\Ascr,\Ascr'\subseteq\binom{[p+q]}{p}$ be not balanced.
Then~\refeq{S_pluck} with these $\Ascr,\Ascr'$ is violated for some $(G,w)$
already in case $\mathfrak{S}=\Rset$ and $n=p+q$ (and hence $X=\emptyset$ and
$Y=[p+q]$). More precisely, there exist a planar network $G=(V,E)$ with $p+q$
sources and a weighting $w:V\to\Rset$ such that the flow-generated function
$f=f_w$ on $2^{[p+q]}$ determined by $w$ gives
  \begin{equation} \label{eq:non-balan}
  \sum\nolimits_{A\in\Ascr} f(A)f(\bar A) \ne
      \sum\nolimits_{A\in\Ascr'} f(A)f(\bar A).
  \end{equation}
  \end{prop}

  \begin{proof}
~Since $\Ascr,\Ascr'$ are not balanced, there exists a nested matching $M$ in
$[p+q]$ (with $|M|=q$) such that
   \begin{equation} \label{eq:diff_M}
   |\Ascr_M|\ne |\Ascr'_M|,
   \end{equation}
where $\Ascr_M$ denotes the set of members $A\in\Ascr$ having $M$ as a feasible
matching: $M\in\Mscr(A)$, and similarly for $\Ascr'$.

We fix such an $M$, and our aim is to construct a planar network $G=(V,E)$ with
$p+q$ sources that satisfies the following properties:
  \begin{itemize}
\item[(P1)] for each $A\in\binom{[p+q]}{p}$ with $M\in\Mscr(A)$, ~$G$ has a
unique $A$-flow and a unique $\bar A$-flow, i.e. $|\Phi_A|=|\Phi_{\bar A}|=1$;
\item[(P2)] if $A\in\binom{[p+q]}{p}$ and $M\notin\Mscr(A)$, then at least one
of $\Phi_A$ and $\Phi_{\bar A}$ is empty.
  \end{itemize}

Once we are given such a $G$, assign $w(v):=1$ for all $v\in V$. In view
of~(P1) and~(P2), for $f=f_w$ and $A\in\binom{[p+q]}{p}$, we have $f(A)=f(\bar
A)=1$ if $M\in\Mscr(A)$, and $f(A)=f(\bar A)=0$ otherwise (equivalently, the
term $f(A)f(\bar A)$ vanishes in the latter case). This implies
$\sum_{A\in\Ascr} f(A)f(\bar A)=|\Ascr_M|$ and $\sum_{A\in\Ascr'} f(A)f(\bar
A)=|\Ascr'_M|$, and now the required inequality~\refeq{non-balan} follows
from~\refeq{diff_M}.
\smallskip

We first construct the desired network $G$ in case $p=q$. This network is
embedded in the upper half-plane $\Rset\times \Rset_{\ge 0}$ and its sources
$s_i$ are identified with the points $(i,0)$, $i=1,\ldots,2p$. The remaining
part of $G$ is designed as follows.

If two elements (arcs) $\pi,\pi'\in M$ obey $[\pi]\supset[\pi']$ and if there
is no $\pi''\in M$ such that $[\pi]\supset[\pi'']\supset[\pi']$, we say that
$\pi$ is the \emph{immediate predecessor} of $\pi'$ and that $\pi'$ is an
\emph{immediate successor} of $\pi$. An arc in $M$ having no predecessor is
called \emph{maximal} (so the intervals of maximal arcs are pairwise disjoint
and their union is $[2p]$).

Consider an arc $\pi=(i,j)$ and let $\Delta=\Delta(\pi):=(j-i+1)/2$. We
represent $\pi$ by graph $C_\pi$ consisting of $2\Delta+1$ vertices and
$2\Delta$ edges whose union forms the half-circumference $\xi_\pi$ lying in the
upper half-plane, connecting the points $s_i$ and $s_j$ and having the center
at the point $(\frac{i+j}{2},0)$. More precisely, the vertices of $C_\pi$ lie
on $\xi_\pi$ and are labeled as $s_i=v_0,u_1,v_1,u_2,\ldots,
v_{\Delta-1},u_\Delta, v_\Delta=s_j$, in this order from left to right, and the
directed edges (formed by the pieces of $\xi_\pi$ between consecutive vertices)
correspond to the pairs $(v_{\ell-1},u_\ell)$ and $(v_\ell,u_\ell)$ for
$\ell=1,\ldots,\Delta$. To indicate the arc in $M$ generating these vertices,
we also write $u_\ell^\pi$ for $u_\ell$, and $v_\ell^\pi$ for $v_\ell$.

Let $G'$ be the (disjoint) union of graphs $C_\pi$, $\pi\in M$. The desired
network $G$ is obtained from $G'$ by drawing additional edges connecting
subgraphs $C_\pi$ and $C_{\pi'}$ for each non-maximal $\pi'$ and its immediate
predecessor $\pi$. More precisely, let $\pi=(i,j)\in M$ and let
$\pi_\alpha=(i_\alpha,j_\alpha)$, $\alpha=1,\ldots,k$, be the immediate
successors of $\pi$ (possibly $k=0$); we assume that the successors are indexed
from left to right, i.e. $i_\alpha=j_{\alpha-1}+1$ (then $i_1=i+1$ and
$j_k=j-1$). Observe that
   $$
   \sum\nolimits_{\alpha=1}^k \Delta(\pi_\alpha)=\Delta(\pi)-1.
   $$
So the total number of $u$-vertices in the subgraphs $C_{\pi_\alpha}$
($\alpha=1,\ldots,k$) is equal to the number of $v$-vertices in $C_\pi$
different from the endvertices $v_0^\pi=s_i$ and $v_{\Delta(\pi)}^\pi=s_j$.
Moving from left to right (and preserving the planarity), we connect these
vertices by $\Delta(\pi)-1$ edges directed from $u$- to $v$-vertices. Formally,
for $\alpha=1,\ldots,k$ and $\ell=1, \ldots,\Delta(\pi_\alpha)$, we draw edge
$e_{\pi_\alpha,\ell}$ from $u_\ell^{\pi_\alpha}$ to
$v^\pi_{(i_\alpha-i-1)/2+\ell}$.

Finally, let $\pi_1,\ldots,\pi_h$ be the maximal arcs in $M$, in this order
from left to right. Then $\Delta(\pi_1)+\ldots +\Delta(\pi_h)=p$. The
concatenation (from left to right) of the sequences of $u$-vertices in
$C_{\pi_1},\ldots,C_{\pi_h}$ is
  $$
  u_1^{\pi_1},\ldots,u_{\Delta(\pi_1)}^{\pi_1},
   u_1^{\pi_2},\ldots,u_{\Delta(\pi_2)}^{\pi_2}, \ldots,
     u_1^{\pi_h},\ldots,u_{\Delta(\pi_h)}^{\pi_h}.
     $$
We assign these $p$ vertices to be the sinks of $G$, denoted as
$t_1,\ldots,t_p$, respectively. This completes the construction of $G$ with
specified sources and sinks. The picture illustrates an example of $G$; here
$p=5$ and $M=\{16,23,45,(7,10),89\}$.

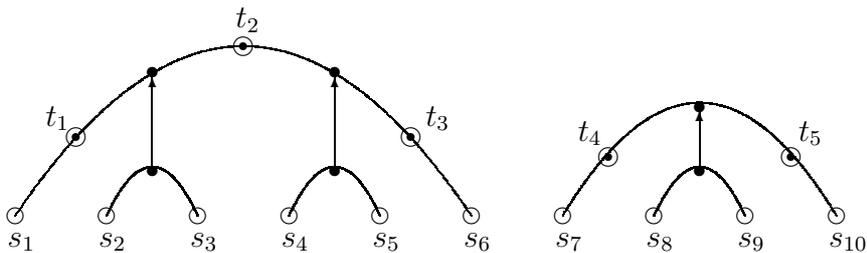
 \begin{center}
  \unitlength=1mm
  \begin{picture}(120,35)

 \put(0,5){\circle{2}}
 \put(12,5){\circle{2}}
 \put(24,5){\circle{2}}
 \put(36,5){\circle{2}}
 \put(48,5){\circle{2}}
 \put(60,5){\circle{2}}
 \put(72,5){\circle{2}}
 \put(84,5){\circle{2}}
 \put(96,5){\circle{2}}
 \put(108,5){\circle{2}}
  \put(-1,1){$s_1$}
  \put(11,1){$s_2$}
  \put(23,1){$s_3$}
  \put(35,1){$s_4$}
  \put(47,1){$s_5$}
  \put(59,1){$s_6$}
  \put(71,1){$s_7$}
  \put(83,1){$s_8$}
  \put(95,1){$s_9$}
  \put(107,1){$s_{10}$}
  \qbezier(0,5)(30,50)(60,5)
  \qbezier(12,5)(18,18)(24,5)
  \qbezier(36,5)(42,18)(48,5)
  \qbezier(72,5)(90,35)(108,5)
  \qbezier(84,5)(90,18)(96,5)
 \put(18,11){\circle*{1.5}}
 \put(18,24){\circle*{1.5}}
 \put(42,11){\circle*{1.5}}
 \put(42,24){\circle*{1.5}}
 \put(90,11){\circle*{1.5}}
 \put(90,19.5){\circle*{1.5}}
  \put(18,11){\vector(0,1){12.5}}
  \put(42,11){\vector(0,1){12.5}}
  \put(90,11){\vector(0,1){8}}
 \put(8,15.5){\circle*{1.2}}
 \put(30,27.5){\circle*{1.2}}
 \put(52,15.5){\circle*{1.2}}
 \put(78,12.8){\circle*{1.2}}
 \put(102,12.8){\circle*{1.2}}
 \put(8,15.5){\circle{2.4}}
 \put(30,27.5){\circle{2.4}}
 \put(52,15.5){\circle{2.4}}
 \put(78,12.8){\circle{2.4}}
 \put(102,12.8){\circle{2.4}}
 \put(4,17){$t_1$}
 \put(29,30){$t_2$}
 \put(54,17){$t_3$}
 \put(74,15){$t_4$}
 \put(103,15){$t_5$}
   \end{picture}
   \end{center}

 \noindent
\textbf{Remark 3.} ~The constructed $G$ is a forest whose connected components
(trees) correspond to maximal arcs in $M$. One can make it connected, e.g., by
adding extra nodes $z_i$ and edges $(z_i,s_i),(z_i,s_{i+1})$ for $i=1,\ldots,
2p-1$. Then the sources and sinks become lying in the boundary of the outer
face and go there in the order $s_{2p},\ldots, s_1,t_1,\ldots, t_p$. In fact,
it does not matter that our network has $p$ (rather than $2p$) sinks; however,
to be consistent with settings in Section~\ref{sec:intr}, we can slightly
modify the network so as to add sinks $t_{p+1},\ldots,t_{2p}$ in a due way.
\medskip

We assert that $G$ satisfies properties~(P1) and~(P2) for the given $M$. To
show~(P1), we use induction on $|M|$. Consider $A\in\binom{[2p]}{p}$ such that
$M\in\Mscr(A)$. Let $\pi=(1,r)$ be the first maximal arc in $M$, and let
$M':=M-\{\pi\}$. By our construction, the corresponding network $G'$ for $M'$
is obtained from $G$ by removing the subgraph $C_\pi$ (and the edges connecting
it with the rest of $G$). It has $p-1$ sinks $t'_1,\ldots,t'_{p-1}$, of which
the last $p-\Delta(\pi)$ sinks are sinks of $G$ (namely, $t'_{\ell}=t_{\ell+1}$
for $\ell=p-\Delta(\pi),\ldots,p-1$), and the first $\Delta(\pi)-1$ sinks
$t'_1,\ldots,t'_{\Delta(\pi)-1}$ are the $u$-vertices in the subgraphs
$C_{\pi'}$ generated by the immediate successors $\pi'$ of $\pi$ in $M$.

The $(p-1)$-element sets $A':=A-\{1,r\}$ and $\tilde A':=\bar A-\{1,r\}$ give a
partition of the set $[2p]-\{1,r\}$ (whose elements can be renumbered as
$1',2',\ldots$ if wished), and both have $M'$ as a feasible matching. By
induction $G'$ has a unique $A'$-flow $F'$ and a unique $\tilde A'$-flow
$\tilde F'$. Observe that $G$ has a unique collection $\Pscr$ of $\Delta(\pi)$
pairwise disjoint directed paths going, respectively, from the vertices
$s_1,t'_1,\ldots,t'_{\Delta(\pi)-1}$ to the vertices
$t_1,\ldots,t_{\Delta(\pi)}$ (namely, the path from $s_1$ to $t_1$ consists of
one edge, and the path from $t'_\ell$ to $t_{\ell+1}$ goes through the vertex
$v_\ell$). Similarly, $G$ has a unique collection $\Qscr$ of pairwise disjoint
directed paths going, respectively, from $t'_1,\ldots,t'_{\Delta(\pi)-1},s_r$
to $t_1,\ldots,t_{\Delta(\pi)}$. Now the (unique) $A$-flow $\Fscr$ and $\bar
A$-flow $\bar\Fscr$, as required in~(P1) for $G$, are obtained by combining
$\Fscr',\tilde \Fscr'$ with $\Pscr,\Qscr$ in a natural way. (In case
$A=A'\cup\{1\}$, $\Fscr$ is formed from $\Fscr',\Pscr$, and $\bar \Fscr$ from
$\tilde \Fscr',\Qscr$. In case $A=A'\cup\{r\}$, $\Fscr$ is formed from
$\Fscr',\Qscr$, and $\bar\Fscr$ from $\tilde \Fscr',\Pscr$.)

Next we show~(P2). Let $A\in\binom{[2p]}{p}$ and $M\not\in\Mscr(A)$. Then there
exists $\pi=(i,j)\in M$ such that one of $A,\bar A$ contains both $i,j$. Choose
$\pi$ with $j-i$ minimum under this property. Assume that $i,j\in A$. Since
each arc $\pi'\in M$ with $[\pi']\subset[\pi]$ has exactly one end in $A$ and
since $i,j\in A$, we observe that the interval $[\pi]$ contains $\Delta(\pi)+1$
elements of $A$. On the other hand, the number of $u$-vertices in the subgraph
$C_\pi$ is only $\Delta(\pi)$ (and removing these vertices from $G$ disconnects
the sources $s_k$ with $i\le k\le j$ from the sinks). Hence no $A$-flow in $G$
can exist. Similarly, when $i,j\in\bar A$, the network $G$ has no $\bar
A$-flow. This yields~(P2).

It remains to consider the case $p>q$. We reduce it to the previous case as
follows. Given a nested matching $M$ (of size $q$) in $[p+q]$
satisfying~\refeq{diff_M}, let $i_1<\ldots<i_{p-q}$ be the sequence of free
elements for $M$. We involve these elements in $p-q$ new arcs
$\pi_\ell=(i_\ell,2p-\ell+1)$, $\ell=1,\ldots,p-q$, which are added to $M$. One
can see that the resulting arc set $\hat M$ is a correct nested matching of
size $p$ in $[2p]$. Accordingly, for each $A\in\binom{[p+q]}{p}$, the partition
$(A,\bar A)$ of $[p+q]$ is associated with the partition $(A,\hat A)$ of
$[2p]$, where $\hat A:=\bar A\cup\{p+q+1,\ldots,2p\}$. Let $G$ be the network
constructed as above for the matching $\hat M$ in $[2p]$, and let
$s_1,\ldots,s_{2p}$ and $t_1,\ldots,t_p$ be the sequences of sources and sinks
in it, respectively. It is not difficult to check that $G$ satisfies the
required properties~(P1) and~(P2) for the initial matching $M$ as well. (If
wished, one can remove the sources $s_{p+q+1},\ldots,s_{2p}$ from $G$.)

This completes the proof of the proposition.
  \end{proof}

%----------------------- Sec.6

\section{\Large The standard basis and Laurent phenomenon}  \label{sec:laurent}

In this section we assume that $\frakS$ is a commutative semiring with
division, i.e. $\frakS$ contains $\underline 1$ and the operation $\odot$ is
invertible (in other words, $(\frakS,\odot)$ is an abelian group). Two
important special cases, mentioned in the Introduction, are: the set
$\Rset_{>}$ of positive reals; the tropicalization $\frakL^{\rm trop}$ of a
totally ordered abelian group $\frakL$, in particular, the set $\Rset_{\max}$
of reals with operations $\oplus=\max$ and $\odot=+$. In these special cases
the corresponding sets $\bf{FG}$ of flow-generated functions on $2^{[n]}$
(namely, $\bf{FG}_n(\Rset_{>})$ and $\bf{FG}_n(\frakL^{\rm trop})$) possess the
following nice properties: (i) all these functions $f$ can be generated by
flows in one planar network, namely, in the half-grid $\Gamma=\Gamma_n$; (ii)
$\bf{FG}$ has as a basis the set $\Iscr_n$ of intervals in $[n]$ (called the
\emph{standard} basis for $\bf{FG}$), and (iii) the values of $f$ are expressed
as (algebraic or tropical) Laurent polynomials in its values on $\Iscr_n$.
These facts are discussed in~\cite{DKK1} (mostly for $\Rset_{\max}$) and
in~\cite{BFZ,FZ} (concerning~(iii)); in essence, the arguments can be directly
extended to an arbitrary $\frakS$ as above. Below we give a brief outline
(which is sufficient to restore the details with help of~\cite{DKK1}).

An important feature of this $\Gamma=(V,E)$ is that for any nonempty interval
$I=[q..r]$ in $[n]$, there exists exactly one feasible flow $\phi_I$ from $S_I$
to the sinks $t_1,\ldots,t_{|I|}$; namely, $\phi_I$ goes through the vertices
$(i,j)$ occurring in the rectangle $[r]\times[r-q+1]$ (more precisely,
satisfying $i\le r$, $j\le r-q+1$ and $i\ge j$). Therefore, given a weighting
$w:V\to \frakS$, the values of $f=f_w$ on the nonempty intervals $[q..r]$ are
viewed as
  \begin{equation} \label{eq:fpq}
  f[q..r]=\bigodot_{j\le i\le r,\; 1\le j\le r-q+1} w(i,j).
  \end{equation}

Note that the number $\frac{n(n+1)}{2}$ of vertices of $\Gamma$ is equal to the
number of nonempty intervals in $[n]$ and the system~\refeq{fpq} is
non-degenerate. So, using division in $\frakS$, denoted as $\odivide$, we can
in turn express the weights of vertices via the values of $f$ on the intervals.
This is computed as
   \begin{equation}\label{eq:wij}
   w(i,j)= \left\{
    \begin{array}{ll}
  (f(I_{i,j})\odot f(I_{i-1,j-1}))\odivide (f(I_{i-1,j})\odot f(I_{i,j-1})) & \mbox{for $i>j$},\\
   f(I_{i,j})\odivide f(I_{i,j-1}) & \mbox{for $i=j$},
  \end{array}
        \right.
   \end{equation}
denoting by $I_{i',j'}$ the interval $[(i'-j'+1).. i']$ and letting
$f(I_{i',0}):=\underline{1}$.

Thus, the correspondence $w\mapsto f_w$ gives a bijection between the set of
weightings $w:V\to\frakS$ and $\frakS^{\Iscr_n^0}$, where $\Iscr_n^0$ is the
set of nonempty intervals in $[n]$. By definition~\refeq{S_f}, the value of
$f=f_w$ on any nonempty subset $A\subseteq [n]$ is represented by a
``polynomial'' in variables $w(v)$, $v\in V$, namely, by a $\oplus$-sum of
products $\odot(w(v)\colon v\in V')$ for some subsets $V'\subseteq V$.
Substituting into this polynomial the corresponding terms from~\refeq{wij}, we
obtain an expression of the form
   $$
   f(A)=\oplus\left( \Pscr_k\colon k=1,\ldots,N\right),
   $$
where each $\Pscr_k$ is a ``monomial'' $\odot(f(I)^{\sigma_k(I)} \colon
I\in\Iscr_n^0)$ with integer (possibly negative) degrees $\sigma_k(I)$. This
means that $f(A)$ is a Laurent polynomial (regarding addition $\oplus$ and
multiplication $\odot$) in variables $f(I)$, $I\in\Iscr_n^0$. \medskip

\noindent\textbf{Remark 4.} ~Analyzing possible flows in $\Gamma$, one can show
that the degrees $\sigma_k(I)$ are bounded and, moreover, belong to
$\{-1,0,1,2\}$. This is proved in~\cite{DKK1} for the tropical case and can be
straightforwardly extended to an arbitrary commutative semiring $\frakS$ with
division. (In~\cite{Sp} similar bounds are established for Laurent polynomials
arising in the octahedron recurrence.) \medskip

Finally, a simple fact (see~\cite{DKK1}) is that any function $f:2^{[n]}\to
\Rset$ obeying TP3-relation~\refeq{TP3} is determined by its values on
$\Iscr_n$. The proof of this fact is directly extended to $\frakS$ in question.
(A hint: if $S\subseteq[n]$ is not an interval, define $i:=\min(S)$,
$k:=\max(S)$, $X:=S-\{i,k\}$, and let $j$ be an element in $[i..k]-S$. Then for
a function $f$ on $2^{[n]}$ obeying SP3-relation~\refeq{SP3}, the value $f(S)$
is expressed via the values $f(S')$ on five sets $S'=Xi,Xj,Xk,Xij,Xjk$. Since
$\max(S')-\min(S')< \max(S)-\min(S)$, we can apply induction on
$\max(S)-\min(S)$.) This fact together with reasonings above implies that
$\Iscr_n$ is a basis for the functions in $\bf{FG}_n(\frakS)$ and that all
these functions are generated by flows in $\Gamma$ (so they are bijective to
weightings $w:V\to\frakS$, up to their values on $\emptyset$, and possess the
Laurentness property as above).

\end{document}